\documentclass[11pt,reqno]{amsart}

\usepackage{lmodern}
\usepackage[T1]{fontenc}
\usepackage[utf8]{inputenc}
\usepackage{microtype}
\usepackage{amssymb}
\usepackage{amsmath}
\usepackage{amsthm}
\usepackage{amsfonts}
\usepackage{amscd}
\usepackage{graphicx}
\usepackage{xcolor}
\usepackage{float}
\usepackage{latexsym}

\newcommand{\seqnum}[1]{\href{https://oeis.org/#1}{#1}}

\usepackage[english]{babel}
\usepackage{mathtools,mathrsfs,booktabs,calc}
\usepackage[shortlabels]{enumitem}
\usepackage{comment,ytableau,csquotes,tikz}
\usetikzlibrary{decorations.pathreplacing}
\newcommand{\interl}{\ll}

\newcommand{\faultfree}{\mathcal{B}}
\newcommand{\anchor}{\mathcal{A}}
\newcommand{\fl}{\mathrm{faultlines}}

\newcommand{\oeis}[1]{\seqnum{#1}}
\definecolor{lightblue}{RGB}{173, 216, 230}
\definecolor{skyblue}{RGB}{135, 206, 255}
\definecolor{pastelblue}{RGB}{176, 224, 230}
\definecolor{iceblue}{RGB}{200, 230, 250}

\definecolor{mintgreen}{RGB}{144, 238, 144}
\definecolor{pastelgreen}{RGB}{152, 251, 152}
\definecolor{softgreen}{RGB}{193, 255, 193} 
\definecolor{lemonyellow}{RGB}{255, 250, 205}
\definecolor{peachpink}{RGB}{255, 218, 185}
\definecolor{lavender}{RGB}{230, 230, 250}
\definecolor{aquablue}{RGB}{175, 238, 238}
\definecolor{softpink}{RGB}{255, 182, 193}
\definecolor{softblue}{RGB}{180, 200, 255}

\newcommand{\defin}[1]{%
\relax\ifmmode%
\textcolor{blue}{#1}%
\else \textcolor{blue}{\emph{#1}}%
\fi%
}

\newcommand{\gap}{\text{\makebox[\widthof{\texttt{0}}]{\texttt{-}}}}

\usepackage{xparse}
\ExplSyntaxOn
\tl_new:N \l_lattice_path_tl
\NewDocumentCommand{\latticePath}{ O{} m m }
 {
  \tl_clear:N \l_lattice_path_tl
  \str_map_inline:nn { #3 }
   {
    \str_case:nnF { ##1 }
     {
       {n}{ \tl_put_right:Nn \l_lattice_path_tl { --~++(0,1)   } }
       {s}{ \tl_put_right:Nn \l_lattice_path_tl { --~++(0,-1)  } }
       {e}{ \tl_put_right:Nn \l_lattice_path_tl { --~++(1,0)   } }
       {w}{ \tl_put_right:Nn \l_lattice_path_tl { --~++(-1,0)  } }
     }
     {}
   }
   \draw[line~width=2.5pt, line~cap=round, line~join=round, #1] (#2) \l_lattice_path_tl ;
 }
\ExplSyntaxOff
\newcommand{\bt}{\mathrm{bigtiles}}

\newcommand{\btw}{\mathrm{width}}
\newcommand{\anchorc}{\mathcal{C}}

\newcommand{\gF}{\mathbf{F}}
\newcommand{\gG}{\mathbf{G}}
\newcommand{\gB}{B}
\newcommand{\gC}{\mathbf{C}}
\newcommand{\gD}{\mathbf{D}}
\newcommand{\gH}{\mathbf{H}}

\newcommand{\drawLoutline}[2]{%
  \draw[thick,fill=skyblue,line width= 1pt]
    (#1,#2) -- ++(2,0) -- ++(0,1) -- ++(-1,0) -- ++(0,1) -- ++(-1,0) -- cycle;
}

\newcommand{\drawbigLoutline}[2]{%
  \draw[thick,fill=skyblue,line width= 1pt]
    (#1,#2) -- ++(3,0) -- ++(0,1) -- ++(-2,0) -- ++(0,2) -- ++(-1,0) -- cycle;
}

\newcommand{\drawRoutline}[2]{%
  \draw[thick,fill=skyblue,line width= 1pt]
    (#1,#2) -- ++(5,0) -- ++(0,2) -- ++(-5,0) -- cycle;
    \fill (0.5+#1,0.5+#2) circle[radius=4pt];
}

\usepackage{hyperref}

\theoremstyle{plain}
\newtheorem{theorem}{Theorem}
\newtheorem{corollary}[theorem]{Corollary}
\newtheorem{lemma}[theorem]{Lemma}
\newtheorem{proposition}[theorem]{Proposition}

\theoremstyle{definition}
\newtheorem{definition}[theorem]{Definition}
\newtheorem{example}[theorem]{Example}

\theoremstyle{remark}
\newtheorem{remark}[theorem]{Remark}

\newtheorem{problem}[theorem]{Problem}

\title{Polynomials from Tilings of Rectangles}

\author{John Ahlberg}
\address{Viktor Rydbergs Gymnasium, Sweden}
\email{john.ahlberg137@gmail.com}

\author{Per Alexandersson}
\address{Department of Mathematics, Stockholm University, 106 91 Stockholm,
  Sweden}
\email{per.w.alexandersson@gmail.com}

\date{}


\begin{document}

\begin{abstract}
We study tilings of rectangular boards using unit squares together
with a single type of big tile shaped as a Ferrers diagram. 
We derive generating functions for these tilings, 
prove real-rootedness and interlacing properties 
of associated independence polynomials, and establish connections 
with several sequences in the OEIS. 
Our results touch on tilings involving L-shaped polyominoes,
fault-free tilings, and cylindric variants.
For rectangular tiles of shape $(k,k)$, we prove real-rootedness
and interlacing in the range $d\leq k$.
\end{abstract}

\maketitle

\section{Introduction}

We consider tilings of a rectangular board with two types of tiles,
the $1{\times}1$ unit cell and one additional \emph{big tile} shaped as a Ferrers diagram~$\mu$.
Only translations of the displayed tile are allowed, with no rotations or reflections.
We study enumerative questions and find new interpretations
of sequences in the OEIS (The Online Encyclopedia of Integer Sequences).
We relate these tilings to restricted words and study the zeros of their
generating polynomials.

Our boards are parametrized by a positive integer~$d$,
the number of rows in which the bottom row of the big tile may be placed.
The board height is then $h + d - 1$,
where $h$ is the height of the big tile.

We consider the generating polynomials keeping track of the 
number of big tiles used---such polynomials can be seen as 
independence polynomials of certain graphs. 
Moreover, we provide efficient generating functions for counting the 
number of tilings with a fixed number of big tiles and fault lines. 
Recent work of Aggarwal, Koley, and Ram~\cite{AggarwalKoleyRam2026}
uses fault-line decompositions to study weighted tilings of
$2k{\times}n$ rectangles by $k{\times}1$ bars.

\subsection{Types of tilings}

We consider several variations of the basic tiling problem:

\begin{itemize}
\item \textbf{Fault-free tilings}:
Tilings without vertical fault lines (see Section~\ref{sec:faultfree}). 
These serve as building blocks for general tilings.

\item \textbf{Dense tilings}: Tilings with no empty column.
This is a relaxation of the fault-free condition, and leads to interesting 
combinatorial properties (see Section~\ref{sec:dense}).

\item \textbf{Cylindric tilings}: 
Tilings on boards with periodic boundary conditions, 
where the leftmost and rightmost columns are identified (see Section~\ref{sec:cylindric}).
\end{itemize}

Each variant leads to distinct generating 
functions and combinatorial properties.

\subsection{Connections with sequences in OEIS}

Our enumerative results recover several known sequences, 
such as the Fibonacci numbers (\oeis{A000045}) and (\oeis{A002478}). 
Additionally, we present several new sequences arising from tilings with 
larger Ferrers shapes and cylindric boundary conditions (see Table~\ref{tab:LCountTable}).

\medskip 

A central theme in our work is the \emph{real-rootedness} of tiling polynomials. 
In particular, we have the following results.
\begin{itemize}
\item For two-column rectangular Ferrers tiles $\mu=(k,k)$ with $d\leq k$,
the tiling polynomials are real-rooted (Theorem~\ref{thm:twoColumnRealRooted}).
\item Cylindric tilings using L-shapes yield real-rooted polynomials 
via the Chudnovsky--Seymour theorem on claw-free graphs (Theorem~\ref{thm:cylindricRealRooted}).
\item Example~\ref{ex:222delayed} gives tiling polynomials
that are real-rooted through $P_{55}$ but fail at $P_{56}$.

\end{itemize}

\subsection{Overview of sections}

We present the material roughly in order of increasing generality,
and the article also serves as a gentle introduction to
enumerative combinatorics and modern tools for studying polynomials.

The paper is organized as follows.
In Section~\ref{sec:L-poly}, we analyze tilings using L-polyominoes and establish the basic recursion.
Section~\ref{sec:anchor} introduces the concept of \emph{anchor words}, which serves as our primary bijective tool.
In Section~\ref{sec:ferrers}, we generalize these results to arbitrary Ferrers diagrams, establish the connection
to independence polynomials, first specialize to two-column Ferrers tiles, and then derive generating functions for hooks and rigid tiles.
We introduce a row-refined tiling polynomial and compare real-rootedness across polynomial variants in Section~\ref{sec:rowRefined}.
Finally, Section~\ref{sec:cylindric} considers the case with cylindric boundary conditions.

\section{Tilings involving \texorpdfstring{$L$}{L}-polyominoes}\label{sec:L-poly}

We gently begin with tilings of the $3{\times}n$-board with the $L$-triomino. 
\begin{example} \label{ex:3x3tiling}
We have the following six tilings for $n=3$:

\begin{equation}
\begin{tikzpicture}[scale=0.4,baseline=(current bounding box.center)]
\draw[step=1cm, thin] (0,0) grid (3,3);
\end{tikzpicture}
\quad
\begin{tikzpicture}[scale=0.4,baseline=(current bounding box.center)]
\draw[step=1cm, thin] (0,0) grid (3,3);
\drawLoutline{0}{1}
\end{tikzpicture}
\quad
\begin{tikzpicture}[scale=0.4,baseline=(current bounding box.center)]
\draw[step=1cm, thin] (0,0) grid (3,3);
\drawLoutline{1}{1}
\end{tikzpicture}
\quad
\begin{tikzpicture}[scale=0.4,baseline=(current bounding box.center)]
\draw[step=1cm, thin] (0,0) grid (3,3);
\drawLoutline{0}{0}
\end{tikzpicture}
\quad
\begin{tikzpicture}[scale=0.4,baseline=(current bounding box.center)]
\draw[step=1cm, thin] (0,0) grid (3,3);
\drawLoutline{1}{0}
\end{tikzpicture}
\quad
\begin{tikzpicture}[scale=0.4,baseline=(current bounding box.center)]
\draw[step=1cm, thin] (0,0) grid (3,3);
\drawLoutline{0}{0}
\drawLoutline{1}{1}
\end{tikzpicture}
\end{equation}
\end{example}

The following theorem gives an interpretation of the sequence \oeis{A002478}. 
This recursion appears in \cite{Deutsch2003} without proof.

\begin{theorem} \label{thm: L-tilings}
   Let $A_n$ be the number of tilings of the $3{\times}n$-board using unit squares and translations of the displayed L-triomino. For $n\geq3$,
   \begin{equation} \label{eq:Lrecursion}
     A_n=A_{n-1}+2A_{n-2}+A_{n-3}, \qquad A_0=1,\, A_1=1,\, A_2=3.
   \end{equation}
\end{theorem}

\begin{proof}
Consider a tiling counted by $A_n$. There are three cases to consider, depending on how the leftmost column is tiled. 
In the first case, every square in the leftmost column is a unit square. 
In the second case, there is an $L$-triomino covering the 
top left corner. In the third case, the lower left corner is covered by an $L$-triomino. 
The first case corresponds to tilings counted by $A_{n-1}$. 
In the second case, the three remaining squares in the first two columns must be $1{\times1}$-squares, so the number of such tilings is $A_{n-2}$. 
The third case is split into two sub-cases as shown in the figure below, and these give $A_{n-2}$ and $A_{n-3}$ as the number of tilings. 

\begin{align*}
\begin{tikzpicture}[scale=0.4,baseline=(current bounding box.center)]
\draw[step=1cm, thin] (0,0) grid (5,3);
\draw[fill=white, line width= 1pt] (0,0) rectangle (1,1);
\draw[fill=white,line width= 1pt] (0,1) rectangle (1,2);
\draw[fill=white,line width= 1pt] (0,2) rectangle (1,3);
\fill[lightgray, opacity=0.5] (1,0) rectangle (5,3);
\end{tikzpicture}
&=
\begin{tikzpicture}[scale=0.4, baseline=(current bounding box.center)]
\draw[step=1cm, thin] (0,0) grid (5,3);
\draw[fill=softgreen, line width= 1pt] (0,0) rectangle (1,1);
\draw[fill=softgreen,line width= 1pt] (0,1) rectangle (1,2);
\draw[fill=softgreen,line width= 1pt] (0,2) rectangle (1,3);
\fill[lightgray, opacity=0.5] (1,0) rectangle (5,3);
\end{tikzpicture}
+
\begin{tikzpicture}[scale=0.4,baseline=(current bounding box.center)]
\draw[step=1cm, thin] (0,0) grid (5,3);
\drawLoutline{0}{1}
\draw[fill=softgreen,line width= 1pt] (1,2) rectangle (2,3);
\draw[fill=softgreen,line width= 1pt] (0,0) rectangle (1,1);
\draw[fill=softgreen,line width= 1pt] (1,0) rectangle (2,1);
\fill[lightgray, opacity=0.5] (2,0) rectangle (5,3);
\end{tikzpicture}
\\
&+
\begin{tikzpicture}[scale=0.4, baseline=(current bounding box.center)]
\draw[step=1cm, thin] (0,0) grid (5,3);
\drawLoutline{0}{0}
\draw[fill=softgreen,line width= 1pt] (0,2) rectangle (1,3);
\draw[fill=softgreen,line width= 1pt] (1,1) rectangle (2,2);
\draw[fill=softgreen,line width= 1pt] (1,2) rectangle (2,3);
\fill[lightgray, opacity=0.5] (2,0) rectangle (5,3);
\end{tikzpicture}
+
\begin{tikzpicture}[scale=0.4,baseline=(current bounding box.center)]
\draw[step=1cm, thin] (0,0) grid (5,3);
\drawLoutline{0}{0}
\drawLoutline{1}{1}
\draw[fill=softgreen,line width= 1pt] (0,2) rectangle (1,3);
\draw[fill=softgreen,line width= 1pt] (2,2) rectangle (3,3);
\draw[fill=softgreen,line width= 1pt] (2,0) rectangle (3,1);
\fill[lightgray, opacity=0.5,] (3,0) rectangle (5,3);
\end{tikzpicture}
\end{align*}
These cases now explain the recursion in \eqref{eq:Lrecursion}.
\end{proof}

The sequence \eqref{eq:Lrecursion} has an interpretation as the number of tilings of the $3{\times}n$-board using squares 
of sizes $1{\times}1$, $2{\times}2$ and $3{\times}3$,
as enumerated by Heubach~\cite[Thm.~4]{Heubach1999}. Such tilings are in bijection with the tilings in Theorem~\ref{thm: L-tilings}.
The bijection is given by 

\begin{equation}
\begin{tikzpicture}[scale=0.4,baseline=(current bounding box.center)]
\draw[fill=softgreen, line width= 1pt] (0,0) rectangle (1,1);
\end{tikzpicture}
\longrightarrow
    \begin{tikzpicture}[scale=0.4,baseline=(current bounding box.center)]
\draw[fill=softgreen, line width= 1pt] (0,0) rectangle (1,1);
\end{tikzpicture}\;,
\hspace{1cm}
\begin{tikzpicture}[scale=0.4,baseline=(current bounding box.center)]
\drawLoutline{0}{0}
\draw[fill=softgreen,line width= 1pt] (1,1) rectangle (2,2);
\end{tikzpicture}
\longrightarrow
\begin{tikzpicture}[scale=0.4,baseline=(current bounding box.center)]
\draw[fill=peachpink,line width= 1pt] (0,0) rectangle (2,2);
\end{tikzpicture}\;,
\hspace{1cm}
\begin{tikzpicture}[scale=0.4,baseline=(current bounding box.center)]
\drawLoutline{0}{0}
\drawLoutline{1}{1}
\draw[fill=softgreen,line width= 1pt] (0,2) rectangle (1,3);
\draw[fill=softgreen,line width= 1pt] (2,2) rectangle (3,3);
\draw[fill=softgreen,line width= 1pt] (2,0) rectangle (3,1);
\end{tikzpicture}
\longrightarrow
\begin{tikzpicture}[scale=0.4,baseline=(current bounding box.center)]
\draw[fill=lemonyellow,line width= 1pt] (0,0) rectangle (3,3);
\end{tikzpicture}.
\end{equation}

\subsection{Tilings with tall L-polyominoes} 

The \defin{$L_h$-polyomino} is the L-polyomino with width 2 
and height $h$. For example, below we have $L_2, L_3$ and $L_4$:

\begin{equation}
\begin{tikzpicture}[scale=0.4]
        \draw[thick,fill=skyblue,line width= 1pt]
    (0,0) -- ++(2,0) -- ++(0,1) -- ++(-1,0) -- ++(0,1) -- ++(-1,0) -- cycle;
     \fill (0.5,0.5) circle[radius=4pt];
\end{tikzpicture}
\quad
\begin{tikzpicture}[scale=0.4]
        \draw[thick,fill=skyblue,line width= 1pt]
    (0,0) -- ++(2,0) -- ++(0,1) -- ++(-1,0) -- ++(0,2) -- ++(-1,0) -- cycle;
     \fill (0.5,0.5) circle[radius=4pt];
\end{tikzpicture}
\quad
\begin{tikzpicture}[scale=0.4]
        \draw[thick,fill=skyblue,line width= 1pt]
    (0,0) -- ++(2,0) -- ++(0,1) -- ++(-1,0) -- ++(0,3) -- ++(-1,0) -- cycle;
    \fill (0.5,0.5) circle[radius=4pt];
\end{tikzpicture}
\end{equation}
All tiles considered in this article have a unique
square appearing in the lower left-hand corner,
and this is called the \defin{anchor square} of that tile.

\begin{lemma} \label{lem:anchor}
    Let $h\geq k\geq 1$ and consider a board with height $2h-1$ and width at least $k+1$. 
    Then the number of ways to place $k$ $L_h$-polyominoes with the anchor squares in the first $k$ columns on this board is $\binom{h}{k}$.
\end{lemma}

\begin{proof}
We first note that there must be exactly one anchor square in each of the $k$ columns, since there can be at most one anchor square in a single column. Moreover, the anchor squares must appear in increasing order from left to right, in order to avoid overlapping tiles. 
Furthermore, the anchor squares can only appear in the first $h$ rows, since the height of the board is $2h-1$. By choosing the rows where the anchor squares are placed we obtain a tiling, see Figure~\ref{fig:Binomiallemma}. This explains the binomial coefficient.
\begin{figure}[!ht]
\centering
\begin{tikzpicture}[scale=0.4,baseline=(current bounding box.center)]
\draw[thick,fill=skyblue,line width= 1.5pt]
(0,0) -- ++(2,0) -- ++(0,1) -- ++(-1,0) -- ++(0,5) -- ++(-1,0) -- cycle;
 \fill (0.5,0.5) circle[radius=4pt];
\draw[thick,fill=skyblue,line width= 1.5pt]
(1,2) -- ++(2,0) -- ++(0,1) -- ++(-1,0) -- ++(0,5) -- ++(-1,0) -- cycle;
\fill (1.5,2.5) circle[radius=4pt];
\draw[thick,fill=skyblue,line width= 1.5pt]
(2,3) -- ++(2,0) -- ++(0,1) -- ++(-1,0) -- ++(0,5) -- ++(-1,0) -- cycle;
\fill (2.5,3.5) circle[radius=4pt];
\draw[thick,fill=skyblue,line width= 1.5pt]
(3,5) -- ++(2,0) -- ++(0,1) -- ++(-1,0) -- ++(0,5) -- ++(-1,0) -- cycle; \fill (3.5,5.5) circle[radius=4pt];
\draw[step=1cm, thin] (0,0) grid (5,11);
\node at (-1,0.5) {1};
\node at (-1,1.5) {2};
\node at (-1,2.5) {3};
\node at (-1,3.5) {4};
\node at (-1,4.5) {5};
\node at (-1,5.5) {6};
\end{tikzpicture}
\caption{This shows a tiling for the case $k=4$, $h=6$. We need to choose four out of the six first rows, where the anchor squares are placed. In total there are $\binom{6}{4} = 15$ tilings.}
\label{fig:Binomiallemma}
\end{figure}
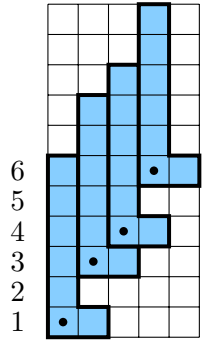
\end{proof}

Given a tiling $T$, we let $\defin{\bt(T)}$
denote the number of tiles which are not the unit square tiles.

A \defin{vertical fault line} is a boundary between two consecutive columns
that is not crossed by a big tile. A tiling is \defin{fault-free} if it has
no vertical fault line.

\begin{definition} \label{def:tallpolynomial}
 Let \defin{$\mathcal{T}_h(d,n)$} be the set of tilings of the $(h{+}d{-}1)\times n$-board using unit squares and $L_h$-polyominoes, where $h$ is the height of the tile
 and $d$ is the number of rows in which an anchor square can be placed.
 We assume $d\leq h$, so that the board has height at most $2h-1$. This ensures there is at most one anchor square in each column.
 
 Let \defin{$P_{d,n}(t)$} be the generating polynomial which tracks the number of $L_h$-polyominoes used. That is,
 \begin{equation} \label{eq:tallLrecursion}
 \defin{P_{d,n}(t)} \coloneqq \sum_{T\in \mathcal{T}_h(d,n)} t^{\bt(T)}.
 \end{equation}
 \end{definition}

For example, $P_{2,3}(t)=1+4t+t^2$ as shown in Example~\ref{ex:3x3tiling}.
\begin{theorem}
 For $1\leq d \leq h$ and $n\geq0$, with $P_{d,m}=0$ for $m<0$,
 we have the linear recursion
\begin{equation}\label{eq:umbralQ}
P_{d,n+1}(t)=\sum_{k=0}^d \binom{d}{k}t^k\cdot P_{d,n-k}(t), \qquad P_{d,0}(t)=1.
\end{equation}
Note that this recursion is independent of the height $h$ of the tile, as long as $h\geq d$.
\end{theorem}
\begin{proof}
Consider a tiling counted by $P_{d,n+1}(t)$.
The first vertical fault line occurs after column~$k+1$,
where $k$ is the number of anchor squares in the initial segment
(see~\eqref{eq:ltlingRecursion}); note that $0\leq k\leq d$.
The initial segment is a fault-free tiling of width~$k+1$,
and by Lemma~\ref{lem:anchor} there are $\binom{d}{k}$ such tilings,
each using $k$ big tiles.
Everything after the fault line is an independent tiling counted by $P_{d,n-k}(t)$.

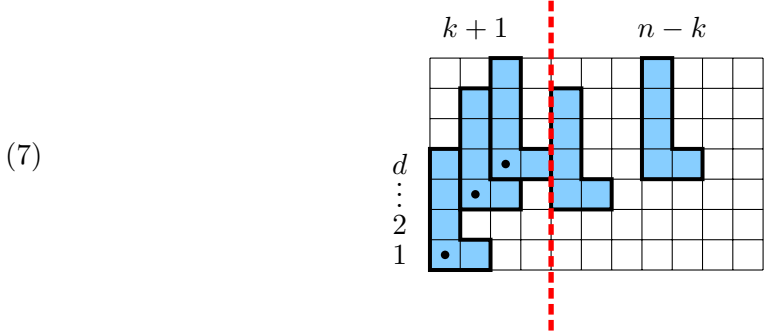
\begin{figure}[ht]
\centering
\begin{equation}\label{eq:ltlingRecursion}
\begin{tikzpicture}[scale=0.4,baseline=(current bounding box.center)]
\draw[thick,fill=skyblue,line width= 1.5pt]
(0,0) -- ++(2,0) -- ++(0,1) -- ++(-1,0) -- ++(0,3) -- ++(-1,0) -- cycle;
 \fill (0.5,0.5) circle[radius=4pt];
\draw[thick,fill=skyblue,line width= 1.5pt]
(1,2) -- ++(2,0) -- ++(0,1) -- ++(-1,0) -- ++(0,3) -- ++(-1,0) -- cycle;
\fill (1.5,2.5) circle[radius=4pt];
\draw[thick,fill=skyblue,line width= 1.5pt]
(2,3) -- ++(2,0) -- ++(0,1) -- ++(-1,0) -- ++(0,3) -- ++(-1,0) -- cycle;
\fill (2.5,3.5) circle[radius=4pt];
\draw[thick,fill=skyblue,line width= 1.5pt]
(4,2) -- ++(2,0) -- ++(0,1) -- ++(-1,0) -- ++(0,3) -- ++(-1,0) -- cycle; 
\draw[thick,fill=skyblue,line width= 1.5pt]
(7,3) -- ++(2,0) -- ++(0,1) -- ++(-1,0) -- ++(0,3) -- ++(-1,0) -- cycle; 
\draw[step=1cm, thin] (0,0) grid (11,7);
\node at (-1,0.5) {1};
\node at (-1,1.5) {2};
\node at (-1,2.75) {$\vdots$};
\node at (-1,3.5) {$d$};
\node at (1.5,8) {$k+1$};
\node at (8,8) {$n-k$};
\draw[dash pattern=on 4pt off 2pt, red, line width=2pt] (4,-2) -- (4,9);
\end{tikzpicture}
\end{equation}
\caption{A tiling in which the first vertical fault line occurs after column $k+1$. The red dashed line marks this fault line, separating an initial fault-free segment of width $k+1$ from an independent tiling of width $n-k$ to its right.}
\end{figure}
\end{proof}

\begin{remark}
The recursion in \eqref{eq:umbralQ} has a close resemblance to the recursion for the 
\emph{Touchard polynomials}
\begin{equation*}
T_{n+1}(x) = \sum_{k=0}^n \binom{n}{k} x \cdot  T_{n-k}(x), \quad T_0(x)=1,
\end{equation*}
where the coefficients are the Stirling numbers of the second kind,
see \oeis{A008277}.
\end{remark}

\section{Words from tilings with \texorpdfstring{$L$}{L}-polyominoes}\label{sec:anchor}

The notion of anchor squares is a helpful tool for analyzing tilings. This motivates the following definition.
\begin{definition} \label{def:anchorword}
An \defin{anchor word} is a word $a_1,a_2,\dots, a_n$ in the alphabet $\{1,2,\dots, d, \infty\}$ such that $a_n=\infty$ and
\begin{equation}
a_i\neq \infty \implies a_{i+1}\geq a_i+1\text{ for all } i=1,2\dots , n-1.
\end{equation}

Let \defin{$\anchor(d,n)$} denote the set of such anchor words for $n\geq1$.
We set $\anchor(d,0)=\{\varepsilon\}$, where $\varepsilon$ is the empty word.
\end{definition}
Throughout the paper, the parameter $d$ denotes the number of rows
in which an anchor square may be placed.
For a tile with $h$ rows, the board has height $h+d-1$;
the $(2h{-}1)\times n$ board from Definition~\ref{def:tallpolynomial}
is the special case $d=h$. 

\begin{example}
Here are all the elements in $\anchor(3,5)$ where a dash represents $\infty$:
\[
\begin{array}{llllll}
\gap\gap\gap\gap\gap & \gap\gap\gap\mathtt{1}\gap & \gap\gap\gap\mathtt{2}\gap & \gap\gap\gap\mathtt{3}\gap & \gap\gap\mathtt{1}\gap\gap & \gap\gap\mathtt{12}\gap \\
\gap\gap\mathtt{13}\gap & \gap\gap\mathtt{2}\gap\gap & \gap\gap\mathtt{23}\gap & \gap\gap\mathtt{3}\gap\gap & \gap\mathtt{1}\gap\gap\gap & \gap\mathtt{1}\gap\mathtt{1}\gap \\
\gap\mathtt{1}\gap\mathtt{2}\gap & \gap\mathtt{1}\gap\mathtt{3}\gap & \gap\mathtt{12}\gap\gap & \gap\mathtt{123}\gap & \gap\mathtt{13}\gap\gap & \gap\mathtt{2}\gap\gap\gap \\
\gap\mathtt{2}\gap\mathtt{1}\gap & \gap\mathtt{2}\gap\mathtt{2}\gap & \gap\mathtt{2}\gap\mathtt{3}\gap & \gap\mathtt{23}\gap\gap & \gap\mathtt{3}\gap\gap\gap & \gap\mathtt{3}\gap\mathtt{1}\gap \\
\gap\mathtt{3}\gap\mathtt{2}\gap & \gap\mathtt{3}\gap\mathtt{3}\gap & \mathtt{1}\gap\gap\gap\gap & \mathtt{1}\gap\gap\mathtt{1}\gap & \mathtt{1}\gap\gap\mathtt{2}\gap & \mathtt{1}\gap\gap\mathtt{3}\gap \\
\mathtt{1}\gap\mathtt{1}\gap\gap & \mathtt{1}\gap\mathtt{12}\gap & \mathtt{1}\gap\mathtt{13}\gap & \mathtt{1}\gap\mathtt{2}\gap\gap & \mathtt{1}\gap\mathtt{23}\gap & \mathtt{1}\gap\mathtt{3}\gap\gap \\
\mathtt{12}\gap\gap\gap & \mathtt{12}\gap\mathtt{1}\gap & \mathtt{12}\gap\mathtt{2}\gap & \mathtt{12}\gap\mathtt{3}\gap & \mathtt{123}\gap\gap & \mathtt{13}\gap\gap\gap \\
\mathtt{13}\gap\mathtt{1}\gap & \mathtt{13}\gap\mathtt{2}\gap & \mathtt{13}\gap\mathtt{3}\gap & \mathtt{2}\gap\gap\gap\gap & \mathtt{2}\gap\gap\mathtt{1}\gap & \mathtt{2}\gap\gap\mathtt{2}\gap \\
\mathtt{2}\gap\gap\mathtt{3}\gap & \mathtt{2}\gap\mathtt{1}\gap\gap & \mathtt{2}\gap\mathtt{12}\gap & \mathtt{2}\gap\mathtt{13}\gap & \mathtt{2}\gap\mathtt{2}\gap\gap & \mathtt{2}\gap\mathtt{23}\gap \\
\mathtt{2}\gap\mathtt{3}\gap\gap & \mathtt{23}\gap\gap\gap & \mathtt{23}\gap\mathtt{1}\gap & \mathtt{23}\gap\mathtt{2}\gap & \mathtt{23}\gap\mathtt{3}\gap & \mathtt{3}\gap\gap\gap\gap \\
\mathtt{3}\gap\gap\mathtt{1}\gap & \mathtt{3}\gap\gap\mathtt{2}\gap & \mathtt{3}\gap\gap\mathtt{3}\gap & \mathtt{3}\gap\mathtt{1}\gap\gap & \mathtt{3}\gap\mathtt{12}\gap & \mathtt{3}\gap\mathtt{13}\gap \\
\mathtt{3}\gap\mathtt{2}\gap\gap & \mathtt{3}\gap\mathtt{23}\gap & \mathtt{3}\gap\mathtt{3}\gap\gap. \\
\end{array}
\]
\end{example}

\begin{lemma} \label{lem:wordbijection}
Let $n\geq 0$, and assume that $1\leq d \leq h$. Then there is a bijection
\begin{equation}
 f:\anchor(d,n) \longrightarrow  \mathcal{T}_h(d,n).
\end{equation}
\end{lemma}

\begin{proof}
For $n=0$, the empty word corresponds to the empty board.
Suppose now that $n\geq1$, $a_1,\dots,a_n\in\anchor(d,n)$ and $h\geq d$.
For every $a_i \neq \infty$ place an $L_h$-polyomino with anchor square in column $i$ and row $a_i$. Every such $L_h$-polyomino is placed on the board with height $h+d-1$, since such a tile can reach at most height $d+h-1\leq 2h-1$. The conditions in Definition~\ref{def:anchorword} ensure that no two polyominoes overlap.
Conversely, record in position~$i$ the anchor row of the tile beginning in
column~$i$, or $\infty$ if there is no such tile. Since $d\leq h$, at most
one tile begins in each column, and non-overlap of tiles beginning in
consecutive columns gives $a_{i+1}\geq a_i+1$; the alternative
$a_i\geq a_{i+1}+h$ is impossible because $a_i\leq d\leq h$.
The last entry is $\infty$ because the big tile has width~2.
This construction is inverse to the one above.
 For example, the word $(1,3,4,\infty,3,\infty,\infty, 4,\infty,1, \infty) \in \anchor(4,11)$ corresponds to the tiling below:

\begin{center}
\begin{tikzpicture}[scale=0.4,baseline=(current bounding box.center)]
\draw[thick,fill=skyblue,line width= 1.5pt]
(0,0) -- ++(2,0) -- ++(0,1) -- ++(-1,0) -- ++(0,3) -- ++(-1,0) -- cycle;
 \fill (0.5,0.5) circle[radius=4pt];
\draw[thick,fill=skyblue,line width= 1.5pt]
(1,2) -- ++(2,0) -- ++(0,1) -- ++(-1,0) -- ++(0,3) -- ++(-1,0) -- cycle;
\fill (1.5,2.5) circle[radius=4pt];
\draw[thick,fill=skyblue,line width= 1.5pt]
(2,3) -- ++(2,0) -- ++(0,1) -- ++(-1,0) -- ++(0,3) -- ++(-1,0) -- cycle;
\fill (2.5,3.5) circle[radius=4pt];
\draw[thick,fill=skyblue,line width= 1.5pt]
(4,2) -- ++(2,0) -- ++(0,1) -- ++(-1,0) -- ++(0,3) -- ++(-1,0) -- cycle; 
\fill (4.5,2.5) circle[radius=4pt];
\draw[thick,fill=skyblue,line width= 1.5pt]
(7,3) -- ++(2,0) -- ++(0,1) -- ++(-1,0) -- ++(0,3) -- ++(-1,0) -- cycle; 
\fill (7.5,3.5) circle[radius=4pt];
\draw[thick,fill=skyblue,line width= 1.5pt]
(9,0) -- ++(2,0) -- ++(0,1) -- ++(-1,0) -- ++(0,3) -- ++(-1,0) -- cycle; 
\fill (9.5,0.5) circle[radius=4pt];
\draw[step=1cm, thin] (0,0) grid (11,7);
\node at (-1,0.5) {1};
\node at (-1,1.5) {2};
\node at (-1,2.5) {3};
\node at (-1,3.5) {4};
\end{tikzpicture}
\end{center}
\end{proof}

This bijection sends the number of integers (not $\infty$) in the anchor word to the number of L-polyominoes. Hence, the polynomial 
\begin{equation}
\defin{Q_{d,n}(t)} \coloneqq \sum_{\substack{w \in \anchor(d,n)}} t^{\bt(w)}
\end{equation}
coincides with the generating polynomial in Definition~\ref{def:tallpolynomial}. Note that \defin{$\bt(w)$} counts the number of integers in $w$.

\begin{corollary}
For $d,n\geq0$, with $Q_{d,0}=1$ and $Q_{d,m}=0$ for $m<0$,
\begin{equation} \label{eq:refined}
Q_{d,n+1}(t)=\sum_{k=0}^d \binom{d}{k}t^k\cdot Q_{d,n-k}(t). 
\end{equation}
\end{corollary}
For $d=3$, we have a connection with multiset $s$-Fibonacci numbers,
see \cite[Thm. 3]{ChuIrmakMillerSzalayZhang2024}, 
and the third row in Table~\ref{tab:LCountTable} below.

\begin{example}
Below is the data for the sizes of $\anchor(d,n)$ as $n=0,1,2, \dots.$

\begin{table}[!ht]
\caption{Number of anchor words in $\anchor(d,n).$}\label{tab:LCountTable}
\footnotesize
\begin{tabular}{@{}cll@{}}
\toprule
$d$ & \textbf{Sequence} & \textbf{OEIS} \\
\midrule
1 & 1, 1, 2, 3, 5, 8, 13, 21, 34, 55, 89, 144, \dots & \oeis{A000045} \\
2 & 1, 1, 3, 6, 13, 28, 60, 129, 277, 595, 1278, 2745, \dots & \oeis{A002478} \\
3 & 1, 1, 4, 10, 26, 69, 181, 476, 1252, 3292, 8657, 22765, \dots & \oeis{A099234} \\
4 & 1, 1, 5, 15, 45, 140, 431, 1326, 4085, 12580, 38740, 119305, \dots & \oeis{A099235} \\
5 & 1, 1, 6, 21, 71, 251, 882, 3088, 10829, 37975, 133146, 466852, \dots & \oeis{A360090} \\
6 & 1, 1, 7, 28, 105, 413, 1624, 6349, 24851, 97315, 380989, 1491567, \dots & New \\
7 & 1, 1, 8, 36, 148, 638, 2766, 11908, 51284, 221049, 952613, 4104980, \dots & New \\
\bottomrule
\end{tabular}
\normalsize
\end{table}
 In general, the generating function for the sequence 
 indexed by $d$ is given by 
\begin{equation}\label{talllgenfunc}
\frac{1}{1-x(x+1)^d}.
\end{equation}
Deleting the final $\infty$ identifies our words with the restricted words
studied by Janji\'c~\cite[Sec.~2]{Janjic2018}, formed from $d$ increasing
letters and one separator.
His Proposition~7 and Corollary~8 give the refinement
\[
Q_{d,n}(t)=\sum_{k=1}^n\binom{dk}{n-k}t^{n-k}\qquad(n\geq1).
\]
Indeed, $k$ counts the increasing runs, including empty runs;
their generating polynomial is $(1+t)^{dk}$.
His Example~9 identifies the cases \oeis{A002478} and \oeis{A099234}.
\end{example}

To refine by the last letter, let $\widehat{\anchor}(d,n)$ be the set of
words satisfying the adjacent-letter inequalities in
Definition~\ref{def:anchorword}, without requiring the last letter to be $\infty$.
For $n\geq1$, set
\begin{equation}
Q_{d,n,j}(t) \coloneqq \sum_{\substack{w \in \widehat{\anchor}(d,n) \\ w_n=j}} t^{\bt(w)}, \qquad j\in\{1,2,\dotsc,d,\infty\}.
\end{equation}
\begin{theorem}
For $n\geq2$, the polynomials $Q_{d,n,j}(t)$ satisfy the following recursions,
with the stated initial values at $n=1$: 
\begin{align}
 Q_{d,n,j}(t) &= 
\begin{cases} Q_{d,n-1,\infty} (t) +
\displaystyle \phantom{t\cdot}\sum_{k=1}^d Q_{d,n-1,k}(t) & \text{ if } j=\infty \\
t\cdot Q_{d,n-1,\infty} (t) +
\displaystyle t \cdot \sum_{k=1}^{j-1} Q_{d,n-1,k}(t) & \text{ if } j < \infty,
\end{cases}
\\
Q_{d,1,j} (t) &=
\begin{cases}
1  & \text{ if } j=\infty \\
t  & \text{ if } j< \infty.
\end{cases}
\end{align}
In other words, 
\begin{equation} \label{eq:matrix1}
\begin{bmatrix}
Q_{d,n,\infty} \\
Q_{d,n,1} \\
Q_{d,n,2} \\
\vdots \\
Q_{d,n,d}
\end{bmatrix}
=
\begin{bmatrix} 
1 & 1 & 1 & \cdots & 1 \\
t & 0 & 0 & \cdots & 0 \\
t & t & 0 & \cdots & 0 \\
\vdots & \vdots & \vdots & \ddots & \vdots \\
t & t & t & \cdots & 0
\end{bmatrix}
\begin{bmatrix}
Q_{d,n-1,\infty} \\
Q_{d,n-1,1} \\
Q_{d,n-1,2} \\
\vdots \\
Q_{d,n-1,d}
\end{bmatrix}.
\end{equation}
\end{theorem}
\begin{proof}
We may append $\infty$ to any word in $\widehat{\anchor}(d,n-1)$.
We may append $j<\infty$ precisely when the previous letter is $\infty$
or a finite letter smaller than $j$. In the latter construction the new
finite letter contributes a factor of $t$.
\end{proof}

\subsection{Interlacing roots}

Our goal now is to show that all the polynomials in \eqref{eq:refined} are \defin{real-rooted}, i.e.,~all their zeros are real numbers. To do this, we introduce the notion of interlacing polynomials.
\begin{definition} 
The following definition is taken from \cite{Wagner1992} and \cite[p.461]{branden2015unimodality}.
Let $f$ and $g$ be polynomials with positive leading coefficients and with real roots $\{f_i\}$ and $\{g_i\}$ respectively. We say that \defin{$f$ interlaces $g$} if $\deg(f)+1=\deg(g)=d$ and 
\[
g_1\leq f_1\leq g_2\leq f_2\leq \cdots \leq f_{d-1}\leq g_d.
\]
Moreover, we say that \defin{$f$ alternates left of $g$} if $\deg(f)=\deg(g)=d$ and 
\[
f_1\leq g_1\leq f_2\leq g_2\leq \cdots \leq f_d\leq g_d.
\]
We say that $f$ \defin{interleaves} $g$ if either $f$ interlaces $g$
or $f$ alternates left of $g$. We write this as $f\interl g$.
By convention, $0\interl 0$, $0\interl h$ and $h\interl 0$
whenever $h$ is a real-rooted polynomial with a positive leading coefficient. 
A sequence $(f_i)_{i=1}^n$ of real-rooted polynomials is called \defin{interlacing} if $f_i \interl f_j$ for all $1\leq i<j\leq n$.
We let \defin{$\mathscr{F}_n^+$} be the family of interlacing sequences of length $n$ with nonnegative coefficients. 
\end{definition}
\begin{example} 
The sequence $(Q_{3,5,j})_{j\in\{\infty,1,2,3\}}$ forms an interlacing sequence, see Table~\ref{tab:Q35zeros}.
\begin{table}[h!]
\centering
\caption{Polynomials \( Q_{3,5,j} \) and their zeros.}\label{tab:Q35zeros}
\begin{tabular}{lll}
\toprule
\text{Polynomial} & \text{Expression} & \text{Zeros (rounded)} \\
\midrule
$Q_{3,5,\infty}$ & $20t^3 + 36t^2 + 12t + 1$ & $\{-1.4,\ -0.27,\ -0.13\}$ \\
$Q_{3,5,1}$ & $t^4 + 15t^3 + 9t^2 + t$ & $\{-14.38,\ -0.47,\ -0.15,\ 0\}$ \\
$Q_{3,5,2}$ & $4t^4 + 21t^3 + 10t^2 + t$ & $\{-4.73,\ -0.38,\ -0.14,\ 0\}$ \\
$Q_{3,5,3}$ & $10t^4 + 28t^3 + 11t^2 + t$ & $\{-2.35,\ -0.31,\ -0.14,\ 0\}$ \\
\bottomrule
\end{tabular}
\end{table}
\end{example}
The following theorem appears in \cite[Thm. 7.8.5]{branden2015unimodality}.

\begin{theorem} \label{thm:ftof}
Let $m,n\geq2$ and let $G$ be an $m{\times}n$ matrix of polynomials in $t$. Then 
$G : \mathscr{F}_n^+ \to \mathscr{F}_m^+$
if and only if
\begin{enumerate}
\item All entries of $G$ have nonnegative coefficients, and
\item For every $2\times2$ sub-matrix 
$\left[\begin{smallmatrix}
a(t)&b(t) \\ c(t)&d(t) 
\end{smallmatrix}\right]$ of $G$ and real numbers 
$\lambda, \mu > 0$, we have
\[
 (\lambda t + \mu) b(t) + d(t) \interl (\lambda t + \mu) a(t) + c(t).
 \]
\end{enumerate}
\end{theorem}

We are now ready to prove that the polynomials arising from the recursion in \eqref{eq:matrix1}
form an interlacing sequence.
\begin{theorem} \label{thm:interlacing}
The sequence $(Q_{d,n,j})_{j\in\{\infty,1,\dotsc,d\}}$ forms an interlacing sequence.
Moreover, $Q_{d,n}(t)=Q_{d,n,\infty}(t)$ for $n\geq1$,
and these polynomials, as well as $Q_{d,0}=1$, have only real zeros. 
\end{theorem}

\begin{proof}
For $d=0$, the vector consists of the constant polynomial~$1$.
We may therefore assume $d\geq1$.
The $2{\times}2$ sub-matrices of the matrix in \eqref{eq:matrix1} are:
\begin{align*}
\begin{array}{ccccc}
\begin{bmatrix}
0 & 0\\
0 & 0
\end{bmatrix},
&
\begin{bmatrix}
0 & 0\\
t & 0
\end{bmatrix},
&
\begin{bmatrix}
1 & 1\\
0 & 0
\end{bmatrix},
&
\begin{bmatrix}
1 & 1\\
t & 0
\end{bmatrix},
&
\begin{bmatrix}
1 & 1\\
t & t
\end{bmatrix},
\\[10pt]
\begin{bmatrix}
t & 0\\
0 & 0
\end{bmatrix},
&
\begin{bmatrix}
t & 0\\
t & 0
\end{bmatrix},
&
{\color{blue}\begin{bmatrix}
t & 0\\
t & t
\end{bmatrix}},
&
\begin{bmatrix}
t & t\\
t & t
\end{bmatrix}.
&
\phantom{\begin{bmatrix}t&t\\t&t\end{bmatrix}}
\end{array}
\end{align*}

By Theorem~\ref{thm:ftof} it is enough to verify that these nine matrices satisfy the conditions. 
We only verify the eighth matrix in detail---the other cases are treated in a similar fashion. 
So in this case the second condition becomes
\[
(\lambda t + \mu) \cdot 0 + t \interl (\lambda t + \mu) t + t \iff t\interl (\lambda t + \mu) t+t.
\] 
This is true since the left-hand side has 0 as a zero, while the right-hand side has 0 and $-\frac{\mu+1}{\lambda}$ as zeros.
\end{proof}

\section{Tiling with Ferrers diagrams}\label{sec:ferrers}

\begin{definition}
Let $\mu = (\mu_0, \mu_1, \mu_2, \dots, \mu_\ell)$ be an integer partition
with positive parts and $\ell\geq1$.
The \defin{Ferrers tile} determined by $\mu$ is 
the polyomino with $\ell+1$ columns of squares where 
column $i$ (indexed from 0) contains $\mu_i$ squares. For example, $\mu= (4,3,3,2,1)$
corresponds to the polyomino below:

\begin{center}
\begin{ytableau}
\ \\
\ & \ & \ \\
\ & \ & \ & \ \\
\ & \ & \ & \ & \ \\
\none[4]&\none[3]& \none[3]&\none[2]& \none[1] \\
\end{ytableau}
\end{center}

We let \defin{$\mathcal{T}_\mu(d,n)$} be the set of tilings of the $(\mu_0{+}d{-}1){\times}n$-board, using unit squares and Ferrers tiles determined by $\mu$, where $d$ is the number of rows available for anchor square placement.
When $d\leq\mu_0$, the board has height at most $2\mu_0-1$,
and two Ferrers tiles cannot have anchor squares in the same column:
if two anchors were placed in the same column, then their rows would differ by at most $d-1\leq \mu_0-1$, while each tile occupies $\mu_0$ consecutive rows in its first column, forcing an overlap.

An \defin{anchor word for $\mu$} is a word $a_1,a_2,\dots, a_n$ in the alphabet $\{1,2,\dots, d, \infty\}$ such that 
for every $i=1,2,\dots, n$, we have the implication
\begin{equation} \label{anchorwordinequalites}
a_i\neq \infty \implies a_{i+k}\geq a_i+\mu_k\quad(1\leq k\leq\ell,\ i+k\leq n),
\end{equation}
and $a_i=\infty$ whenever $1\leq i\leq n$ and $i>n-\ell$.
Thus no tile starts too close to the right boundary. If $n<\ell+1$,
the only word has all letters equal to $\infty$.
Let \defin{$\anchor_{\mu}(d,n)$} denote the set of such anchor words,
and set $\anchor_\mu(d,0)=\{\varepsilon\}$.
\end{definition}

\begin{proposition} \label{anchortotiling}
There is a bijection between $\anchor_{\mu}(d,n)$ and $\mathcal{T}_\mu(d,n)$ whenever $1\leq d \leq \mu_0$.
\end{proposition}
\begin{proof}
Given an anchor word $a_1,\dots,a_n\in \anchor_\mu(d,n)$, place a Ferrers tile with anchor square in column $i$ and row $a_i$ whenever $a_i\neq \infty$.
The conditions in \eqref{anchorwordinequalites} ensure that tiles in different columns do not overlap, and the assumption $d\leq \mu_0$ ensures that no two tiles can have anchor squares in the same column.
Conversely, every tiling in $\mathcal{T}_\mu(d,n)$ determines a unique anchor word by recording the row of the anchor square in each column and writing $\infty$ in columns with no anchor square.
These two constructions are inverse to each other.
\end{proof}
We introduce the generating function
\begin{equation} \label{eq:tilinggf}
\defin{\gF_{\mu,d}(x,t)}\coloneqq
\sum_{n\geq0}\sum_{w\in \anchor_{\mu}(d,n)} t^{\bt(w)}x^n
= \sum_{n\geq0} P_{\mu,d,n}(t)x^n.
\end{equation}
We refer to \defin{$P_{\mu,d,n}(t)$}
as an \defin{anchor-word polynomial}; when $d\leq\mu_0$, it is also
the \defin{tiling polynomial} of $\mathcal T_\mu(d,n)$.
We use $P_{\mu,d,0}=1$ and $P_{\mu,d,m}=0$ for $m<0$ in recurrences.
These definitions generalize Definition~\ref{def:anchorword}.
Observe that $\mu_0$ does not show up in \eqref{anchorwordinequalites}, so we have that
\begin{equation} \label{anchorignorefirstcolumn}
\anchor_{(\mu_0,\mu_1,\dots,\mu_\ell)} (d,n) = \anchor_{(\mu_1,\mu_1,\dots,\mu_\ell)} (d,n).
\end{equation}

For arbitrary $d\geq1$, put
$\widehat\mu=(\max\{\mu_0,d\},\mu_1,\dotsc,\mu_\ell)$.
The anchor words for $\mu$ and $\widehat\mu$ coincide, and
Proposition~\ref{anchortotiling} gives a tiling interpretation using
$\widehat\mu$. All subsequent uses of tilings for an anchor-word polynomial
refer to this representative if $d>\mu_0$.
They need not count all ordinary tilings by the original tile.
For example, $\mu=(2,2)$, $d=3$, $n=2$ gives the anchor-word polynomial
$1+3t$, whereas ordinary tilings of the $4\times2$ board give $1+3t+t^2$.
The extra term comes from two stacked squares.

\begin{lemma}
Suppose $\mu$ and $\nu$ are integer partitions such that $\mu \subseteq \nu$, i.e. $\mu_1 \leq \nu_1, \mu_2 \leq \nu_2, \dots$. Then 
\begin{equation}
|\anchor_\mu (d,n)| \geq |\anchor_\nu (d,n)|.
\end{equation}
\end{lemma}

\begin{proof}
Every anchor word of type $\nu$ is an anchor word of type $\mu$, since the conditions in \eqref{anchorwordinequalites} are more relaxed. Intuitively there are more tilings when using smaller tiles.
\end{proof}

\subsection{Connection with independence polynomials}

The \defin{independence polynomial} of the 
graph $G$, introduced by Gutman and Harary~\cite{GutmanHarary1983}, is defined as
\[
I(G, t) = \sum_{k \geq 0} i_k(G) t^k,
\]
where $i_k(G)$ is the number of independent sets of size $k$. 
An independent set of a graph $G$ is a set of vertices such that no two vertices in the set are adjacent.

The following lemma shows that the polynomials 
in \eqref{eq:tilinggf} have an alternative interpretation.
\begin{lemma}
Every tiling polynomial is an 
\emph{independence polynomial} 
of some graph. 
\end{lemma}
\begin{proof}
Use the representative $\widehat\mu$ above when $d>\mu_0$,
and let $R$ be its rectangular board. 
We construct a graph $G=(V, E)$, called the \defin{conflict graph} (or intersection graph), as follows:
\begin{itemize}
    \item \textbf{Vertices $V$:} Each vertex $v \in V$ corresponds to a unique possible placement of a single big tile on the board $R$.
    \item \textbf{Edges $E$:} An edge $(v_1, v_2) \in E$ exists between two vertices $v_1$ and $v_2$ if and only if their corresponding tile placements on $R$ overlap.
\end{itemize}
The set of non-overlapping big tiles used in a tiling $T$
forms an \defin{independent set} $I$ in $G$. 
Conversely, every independent set $I$ in $G$ corresponds to a unique partial tiling of the board by non-overlapping big tiles. 

The tiling polynomial $P_{\mathcal{T}}(t)$ is therefore 
equal to the independence polynomial $I(G, t)$ 
of the conflict graph $G$.
\end{proof}
For further background on independence polynomials, see the survey by Levit and Mandrescu~\cite{LevitMandrescu2005}.

Note that $P_{\mu,d,n}(t)$ is an independence polynomial
even when $d \geq \mu_0$, but one must add auxiliary
edges (encoding the restrictions in \eqref{anchorwordinequalites}) to the conflict graph compared to
the tiling interpretation in Proposition~\ref{anchortotiling}.

\begin{example}
The following graph is the conflict graph corresponding 
to tilings of the $3{\times}7$ board with the $L$-triomino. 

\begin{center}
\begin{tikzpicture}[scale=1.3, every node/.style={circle, draw, minimum size=6mm, inner sep=0pt, font=\small}]
\foreach \i in {1,...,6} {
  \pgfmathtruncatemacro{\topnum}{2*\i - 1}
  \pgfmathtruncatemacro{\botnum}{2*\i}
  \node (t\i) at (\i, 1) {\topnum};
  \node (b\i) at (\i, 0) {\botnum};
}
\foreach \i in {1,...,5} {
  \pgfmathtruncatemacro{\j}{\i + 1}
  \draw (t\i) -- (t\j);
  \draw (b\i) -- (b\j);
}
\foreach \i in {1,...,6} {
  \draw (t\i) -- (b\i);
}
\foreach \i in {1,...,5} {
  \pgfmathtruncatemacro{\j}{\i + 1}
  \draw[dashed, thick] (t\i) -- (b\j);
}
\end{tikzpicture}
\end{center}
\end{example}

A celebrated result by Chudnovsky and Seymour~\cite{Chudnovsky2007} states that the independence polynomial of a \emph{claw-free graph} has only real zeros.
However, the conflict graphs arising from tilings are in general not claw-free.
The following lemma gives a sufficient condition for the presence of a claw.

\begin{lemma}\label{lem:clawBound}
Let $\mu = (\mu_0,\mu_1,\dotsc,\mu_\ell)$ with $\ell \geq 2$.
If $\mu_{\ell-1}+1\leq d\leq\mu_0$ and $n\geq3\ell+1$,
then the conflict graph contains a claw.
\end{lemma}

\begin{proof}
We exhibit four tile placements whose anchor squares
$u,v_1,v_2,v_3$ form a claw, i.e., $u$ is adjacent to each $v_i$
but the $v_i$ are pairwise non-adjacent.
Set $h=\mu_{\ell-1}$ and place the tiles with anchor squares at
\[
u=(\ell{+}1,\,1),\quad v_1=(1,\,1),\quad v_2=(\ell,\,h{+}1),\quad v_3=(2\ell{+}1,\,1).
\]
\begin{center}
\begin{tikzpicture}[scale=0.35,baseline=(current bounding box.center)]
\draw[fill=skyblue,draw=black,line width=0.8pt]
  (0,0) -- (0,5) -- (1,5) -- (1,4) -- (2,4) -- (2,3) -- (3,3) -- (3,2) -- (4,2) -- (4,0) -- cycle;
\draw[fill=softgreen,draw=black,line width=0.8pt]
  (2,3) -- (2,8) -- (3,8) -- (3,7) -- (4,7) -- (4,6) -- (5,6) -- (5,5) -- (6,5) -- (6,3) -- cycle;
\draw[fill=lavender,draw=black,line width=0.8pt]
  (6,0) -- (6,5) -- (7,5) -- (7,4) -- (8,4) -- (8,3) -- (9,3) -- (9,2) -- (10,2) -- (10,0) -- cycle;
\draw[draw=blue,line width=1.2pt,dashed]
  (3,0) -- (3,5) -- (4,5) -- (4,4) -- (5,4) -- (5,3) -- (6,3) -- (6,2) -- (7,2) -- (7,0) -- cycle;
\fill (0.5,0.5) circle[radius=3pt]; \node[font=\small,below right] at (0.5,0.20) {$v_1$};
\fill (3.5,0.5) circle[radius=3pt]; \node[font=\small,below right] at (3.5,0.20) {$u$};
\fill (2.5,3.5) circle[radius=3pt]; \node[font=\small,above right] at (2.5,3.5) {$v_2$};
\fill (6.5,0.5) circle[radius=3pt]; \node[font=\small,below right] at (6.5,0.20) {$v_3$};
\end{tikzpicture}
\end{center}
All four tiles fit since $h+1\leq d\leq\mu_0$ and $n\geq3\ell+1$.
The tile at $u$ meets $v_1$ and $v_3$ in row~1, and meets $v_2$
in column $\ell+1$, row $h+1$.
In their common columns the tile at $v_1$ has height at most $h$,
whereas $v_2$ starts in row $h+1$.
The column intervals of $v_3$ and each of $v_1,v_2$ are disjoint.
Thus the three leaves are pairwise nonadjacent.
\end{proof}

In the tiling setting we require $d \leq \mu_0$
(see Proposition~\ref{anchortotiling}).
Combined with $d \geq \mu_{\ell-1}+1$, this forces $\mu_0 > \mu_{\ell-1}$,
so this construction applies only to non-rectangular $\mu$.

\begin{problem}
Classify Ferrers tiles $\mu$ and parameters $d$ for which the
conflict graph is claw-free.
\end{problem}

\subsection{Two-column Ferrers tiles}\label{sec:twoColumn}

As noted in \eqref{anchorignorefirstcolumn}, it is enough to study the case $\mu=(k,k)$ for $k\geq 1$.

\begin{lemma}\label{lem:clawFree}
 For $d \leq k$, the conflict graph for $\mu=(k,k)$ is claw-free.
\end{lemma}
\begin{proof}
Recall that a vertex $(i,r)$ corresponds to placing a $k\times 2$ tile with anchor square 
in column $i$ at row $r$, where $1\leq r\leq d$.
Two vertices $(i,r)$ and $(j,r')$ are adjacent in the conflict graph if and only if
their tile placements overlap, which requires $|i-j|\leq 1$ (column intervals overlap)
and $|r-r'|<k$ (row intervals overlap).

Since $d\leq k$, all anchor rows satisfy $r,r'\in\{1,\ldots,d\}\subseteq\{1,\ldots,k\}$,
so $|r-r'|\leq d-1\leq k-1<k$ automatically.
Hence, for any two placements with $d\leq k$, they are adjacent \emph{if and only if} their column
intervals overlap, i.e., $|i-j|\leq 1$.

Now suppose $u=(i,r)$ is the center of a claw with leaves $v_1,v_2,v_3$.
Each $v_p=(j_p,r_p)$ is a neighbor of $u$, so $j_p\in\{i-1,i,i+1\}$.
Two leaves $v_p$ and $v_q$ are non-adjacent if and only if $|j_p-j_q|\geq 2$.
Among three values in the set $\{i-1,i,i+1\}$, at most one pair can have distance $\geq 2$,
namely the pair $\{i-1,i+1\}$.
Therefore, at most one pair among $v_1,v_2,v_3$ can be non-adjacent,
so the three leaves cannot be pairwise non-adjacent.
This contradiction shows that no claw exists.
\end{proof}

\begin{theorem}\label{thm:twoColumnRealRooted}
For $d\leq k$ set $P_n(t) \coloneqq P_{(k,k),d,n} (t)$.
Then the generating function for the $P_0,P_1,P_2, \dots$ is
\begin{equation}
\gF_{(k,k),d} (x,t) = \frac{1}{1-(x+d\cdot tx^2)}.
\end{equation}
Moreover, we have the recursion 
\begin{equation} \label{recursionrectangles}
P_n (t)=P_{n-1} (t) + d\cdot t \cdot P_{n-2} (t)
\iff 
\begin{bmatrix}
P_{n-1} \\
P_{n}
\end{bmatrix}
=
\begin{bmatrix}
0 & 1 \\
dt & 1
\end{bmatrix}
\begin{bmatrix}
P_{n-2} \\
P_{n-1}
\end{bmatrix}
\end{equation}
and all the $P_n (t)$ have only real zeros.
\end{theorem}
\begin{proof}
The generating function identity $\gF_{(k,k),d}(x,t)=1/(1-(x+dtx^2))$ follows from the
observation that every anchor word is built by appending either an empty column (contributing $x$)
or a single big tile in one of the $d$ available rows (contributing $dtx^2$, with the mandatory trailing $\infty$).
The recursion $P_n(t)=P_{n-1}(t)+dt\cdot P_{n-2}(t)$ is immediate.

For real-rootedness, we use Theorem~\ref{thm:ftof}.
The initial pair $(P_0,P_1)=(1,1)$ consists of positive constants,
so it is vacuously interlacing.
It remains to check that the transfer matrix in \eqref{recursionrectangles}
preserves interlacing pairs.
Its only $2{\times}2$ submatrix is the full matrix
\[
\begin{bmatrix}
0 & 1\\
dt & 1
\end{bmatrix}.
\]
Thus, for all $\lambda,\mu>0$, the condition in Theorem~\ref{thm:ftof}
becomes
\[
(\lambda t+\mu)\cdot 1 + 1
\interl
(\lambda t+\mu)\cdot 0 + dt,
\]
or equivalently
\[
\lambda t+\mu+1 \interl dt.
\]
Both polynomials have nonnegative coefficients and real zeros,
with zeros $-(\mu+1)/\lambda$ and $0$, respectively; hence the required
interlacing holds. Therefore $(P_{n-1},P_n)$ is interlacing for every~$n$,
and in particular each $P_n(t)$ is real-rooted.
\end{proof}

The following two results of P.~Br\"and\'en and R.~Leite will be used several
times; see \cite[Thms.~4.4 and~4.5]{BrandenLeite2024x}.

\begin{theorem}[Br\"and\'en--Leite]\label{thm:brandenleite}
Let $f(x)$ be a polynomial with nonnegative coefficients
and all real zeros.
Consider the formal power series
\begin{equation}\label{eq:BLgf}
\frac{1}{1-t(f(x)-f(0))} = \sum_{n\geq 0} r_n(t)\, x^n \in \mathbb{R}[t][[x]].
\end{equation}
Then $r_n(t)$ is real-rooted for each $n\in\mathbb{N}$,
and $r_n(t) \interl r_{n+1}(t)$.
Moreover, if $f(0)\neq 0$, then all zeros of $r_n(t)$ lie in $[-1/f(0),\, 0]$.
\end{theorem}

\begin{theorem}[Br\"and\'en--Leite]\label{thm:brandenleite2}
Let $Q(x)$ be a nonconstant polynomial with $Q(0)>0$
whose zeros are all real and positive, and let $r$ be a positive integer.
Consider the power series
\begin{equation}\label{eq:BLgf2}
\frac{1}{Q(x)-tx^r} = \sum_{n\geq 0} q_n(t)\, x^n.
\end{equation}
Then $q_n(t)$ is a polynomial whose zeros are real and negative for each
$n\in\mathbb{N}$. Moreover, $q_n(t) \interl q_{n+1}(t)$ for each
$n\in\mathbb{N}$.
\end{theorem}

An alternative proof of Theorem~\ref{thm:twoColumnRealRooted} follows from
Theorem~\ref{thm:brandenleite2}.
For $k=d=2$, the unweighted enumeration is Heubach's
$3\times n$ square-tiling case~\cite[Sec.~6.2]{Heubach1999}.
More generally, these polynomials are the matching generating polynomials
of a path on $n$ vertices, with edge activity $dt$.
The generating function for $P_0(t),P_1(t),\dotsc$ is 
\begin{equation}
\frac{1/d}{\tfrac{1}{d}(1-x) - t x^2}.
\end{equation}
This is (up to the factor $\tfrac1d$) of the form 
\[
  \frac{1}{Q(x) - t x^r} \qquad Q(0) >0, \quad r \in \{1,2,\dotsc\}
\]
and $Q(x)$ is real-rooted. Theorem~\ref{thm:brandenleite2} implies that the
polynomials $P_j(t)$ are all real-rooted and satisfy $P_j \interl P_{j+1}$.

\begin{definition}[Unimodality and log-concavity]
Let $p(t) = a_0 + a_1 t + \cdots + a_n t^n$ be a polynomial with nonnegative coefficients.
We say that $p(t)$ is \defin{unimodal} 
if there exists an index $k$ such that
\[
a_0 \leq a_1 \leq \cdots \leq a_k \geq a_{k+1} \geq \cdots \geq a_n.
\]
We say that $p(t)$ is \defin{log-concave} if 
\[
a_i^2 \geq a_{i-1} \cdot a_{i+1} \quad \text{for all } i = 1, 2, \ldots, n-1.
\]
A real-rooted polynomial with nonnegative coefficients is log-concave
and has no internal zeros. A nonnegative log-concave coefficient sequence
with no internal zeros is unimodal. 
However, the converses do not hold in general.
\end{definition}
For further background on these notions and their role in algebraic combinatorics, see Stanley's survey~\cite{Stanley1989} and Br\"and\'en's notes~\cite{branden2015unimodality}.

In particular, we emphasize that the independence 
polynomials for a general graph might not even be unimodal, see \cite{AlaviMaldeSchwenkErdos1987}.

\subsection{Fault-free tilings and fault-free words}\label{sec:faultfree}

Fault-free tilings have a long history; see Graham~\cite{Graham1981}.
For square tiles, Heubach~\cite[Lemma~1]{Heubach1999} gives the convolution
obtained by splitting at vertical fault lines. The word decomposition here
also fits the concatenation framework of
Birmajer, Gil and Weiner~\cite{BirmajerGilWeiner2016}.

\begin{definition}
A \defin{fault-free anchor word} for
$\mu=(\mu_0,\dotsc,\mu_\ell)$ is an anchor word that is either $(\infty)$,
or starts with an integer and contains the block $\infty^\ell$ exactly once,
at its end. We let
\defin{$\faultfree_\mu(d)$} be the set of fault-free anchor words for $\mu$
and $d$, of any length.
\end{definition}

When $d\leq\mu_0$, Proposition~\ref{anchortotiling} shows that these are
precisely the anchor words corresponding to fault-free tilings. For example,
the tiling below is a concatenation of four fault-free tilings separated by
three fault lines.

\begin{center}
\begin{tikzpicture}[scale=0.4,baseline=(current bounding box.center)]
\draw[step=1cm, thin] (0,0) grid (14,5);
\begin{scope}
\drawbigLoutline{0}{0}
\drawbigLoutline{2}{1}
\drawbigLoutline{3}{2}
\end{scope}
\begin{scope}[xshift=7cm, step=1cm] 
\drawbigLoutline{0}{0}
\drawbigLoutline{1}{1}
\end{scope}
\begin{scope}[xshift=11cm, step=1cm] 
\drawbigLoutline{0}{2}
\end{scope}
\draw[dash pattern=on 4pt off 2pt, red, line width=2pt] (6,-1) -- (6,7);
\draw[dash pattern=on 4pt off 2pt, red, line width=2pt] (7,-1) -- (7,7);
\draw[dash pattern=on 4pt off 2pt, red, line width=2pt] (11,-1) -- (11,7);
\end{tikzpicture}
\end{center}
This tiling has the following anchor word that has been split into fault-free words:
\[
1,\infty, 2,3,\infty,\infty \mid\infty\mid 1,2,\infty, \infty\mid 3, \infty,\infty.
\]

This is precisely the replacement of $\mu$ by $\widehat\mu$ introduced above.
Enlarging the first column to height $\max\{\mu_0,d\}$ only prevents two
anchor squares from lying in the same column; it does not change the
inequalities in \eqref{anchorwordinequalites}, and hence does not change the
anchor words.

\begin{lemma} \label{recurrence-length}
Suppose $\mu$ has $\ell +1$ columns. The longest fault-free word in $\faultfree_\mu (d)$ has length given by
\begin{equation} \label{eq:recursionlength}
   \left\lfloor \frac{d-1}{\mu_{\ell}} \right \rfloor \cdot \ell + \ell +1.
\end{equation}
In particular, $\faultfree_\mu (d)$ is finite. 
\end{lemma}

\begin{proof}
The case $d=1$ is easy and there is only one fault-free tile.

We first consider the case $\mu=(h,h,h,\dots, h)$. The following tiling illustrates the maximum length word of a word in $\faultfree_\mu (d)$.

\begin{center}
\begin{tikzpicture}[scale=0.4,baseline=(current bounding box.center)]
\draw[step=1cm, thin] (0,0) grid (21,10);
\drawRoutline{0}{0}
\drawRoutline{4}{2}
\drawRoutline{8}{4}
\drawRoutline{12}{6}
\drawRoutline{16}{8}
\node at (-1, 0.5) {$1$};
\node at (-1, 1.5) {$h$};
\node at (-1, 3.25) {$\vdots$};
\node at (-1, 7.25) {$\vdots$};
\node at (-1, 4.25) {$\vdots$};
\node at (-1, 5.25) {$\vdots$};
\node at (-1, 6.25) {$\vdots$};
\node at (-1, 8.5) {$d$};
\draw [decorate, decoration={brace, mirror, amplitude=5pt}] 
(0,-0.5) -- (5,-0.5) node[midway, below=6pt] {$\ell +1$};
\draw [decorate, decoration={brace, mirror, amplitude=5pt}] 
(17,-0.5) -- (21,-0.5) node[midway, below=6pt] {$\ell$};
\draw [decorate, decoration={brace, mirror, amplitude=5pt}] 
(9,-0.5) -- (12,-0.5) node[midway, below=6pt] {$\ell -1$};
\end{tikzpicture}
\end{center}

As in the proof of Lemma~\ref{lem:anchor}, the anchor squares must appear in increasing order. If $k$ is the total number of big tiles used, then $1+(k-1)h \leq d$. Moreover, the total width is $k\cdot \ell + 1$ (recall that the big tile has width $\ell + 1$). From these two observations the expression in \eqref{eq:recursionlength} follows. The same argument can be made for general $\mu$.

\end{proof}

Every anchor word can be expressed as a concatenation of some fault-free words. In other words, the fault-free words serve as building blocks for general tilings (as long as $d \leq \mu_0$, i.e.\ the board has height $\mu_0+d-1 \leq 2\mu_0 -1$). Given a fault-free word $w \in \faultfree_\mu (d)$ let \defin{$\btw(w)$} be the length of the anchor word.

We let the generating polynomial for the set of fault-free words in $\faultfree_\mu (d)$ be:
\begin{equation}
\defin{\gB_{\mu,d}(x,t)} \coloneqq \sum_{w\in \faultfree_\mu (d)} x^{\btw(w)} \cdot t^{\bt(w)}.
\end{equation}
\begin{theorem}\label{thm:genfuncformula}
The generating function for tilings with a Ferrers tile $\mu$ is
\begin{equation} \label{genfuncformula}
\gF_{\mu,d} (x,t) = \frac{1}{1- \gB_{\mu,d}(x,t)}.
\end{equation}
In particular, the sequence of polynomials $\left\{P_{\mu,d,n}(t)\right\}_{n\geq 0}$ in \eqref{eq:tilinggf} satisfies a linear recursion of length at most $\ell \cdot d +1$, with coefficients in $\mathbb{N}[t]$:
\begin{equation}
P_n=P_{n-1}+ \cdots + \alpha_j (t) P_{n-j}+ \cdots
\end{equation}
where $\alpha_j (t)$ is the sum of fault-free words of length $j$, weighted by $t^{\bt(w)}$.
\end{theorem}
\begin{proof}
This follows from standard theory of generating functions, see for example \cite[Prop. 4.7.11]{Stanley2011}.
\end{proof}

\begin{example}
We now compute $\gB_{\mu,d}(x,t)$ for the Ferrers tile $\mu=(3,1,1)$ and $d=3$. The board height is $\mu_0+d-1=5$, so this corresponds to finding all fault-free tilings on $5 \times n$ boards.
They are: 


\begin{center}
\begin{tikzpicture}[scale=0.4,baseline=(current bounding box.center)]
\draw[step=1cm, thin] (0,0) grid (1,5);
\draw[fill=softgreen,line width= 1pt] (0,0) rectangle (1,1);
\draw[fill=softgreen,line width= 1pt] (0,1) rectangle (1,2);
\draw[fill=softgreen,line width= 1pt] (0,2) rectangle (1,3);
\draw[fill=softgreen,line width= 1pt] (0,3) rectangle (1,4);
\draw[fill=softgreen,line width= 1pt] (0,4) rectangle (1,5);
\node at (0.5,-0.5) {$x$};
\end{tikzpicture}
\quad
\begin{tikzpicture}[scale=0.4,baseline=(current bounding box.center)]
\draw[step=1cm, thin] (0,0) grid (3,5);
\drawbigLoutline{0}{0}
\node at (1.5 ,-0.5) {$tx^3$};
\end{tikzpicture}
\quad
\begin{tikzpicture}[scale=0.4,baseline=(current bounding box.center)]
\draw[step=1cm, thin] (0,0) grid (3,5);
\drawbigLoutline{0}{1}
\node at (1.5 ,-0.5) {$tx^3$};
\end{tikzpicture}
\quad
\begin{tikzpicture}[scale=0.4,baseline=(current bounding box.center)]
\draw[step=1cm, thin] (0,0) grid (3,5);
\drawbigLoutline{0}{2}
\node at (1.5 ,-0.5) {$tx^3$};
\end{tikzpicture}

\begin{tikzpicture}[scale=0.4,baseline=(current bounding box.center)]
\draw[step=1cm, thin] (0,0) grid (4,5);
\drawbigLoutline{0}{0}
\drawbigLoutline{1}{1}
\node at (1.5 ,-0.5) {$t^2x^4$};
\end{tikzpicture}
\quad
\begin{tikzpicture}[scale=0.4,baseline=(current bounding box.center)]
\draw[step=1cm, thin] (0,0) grid (4,5);
\drawbigLoutline{0}{0}
\drawbigLoutline{1}{2}
\node at (1.5 ,-0.5) {$t^2x^4$};
\end{tikzpicture}
\quad
\begin{tikzpicture}[scale=0.4,baseline=(current bounding box.center)]
\draw[step=1cm, thin] (0,0) grid (4,5);
\drawbigLoutline{0}{1}
\drawbigLoutline{1}{2}
\node at (1.5 ,-0.5) {$t^2x^4$};
\end{tikzpicture}
\\
\begin{tikzpicture}[scale=0.4,baseline=(current bounding box.center)]
\draw[step=1cm, thin] (0,0) grid (5,5);
\drawbigLoutline{0}{0}
\drawbigLoutline{2}{1}
\node at (1.5 ,-0.5) {$t^2x^5$};
\end{tikzpicture}
\quad
\begin{tikzpicture}[scale=0.4,baseline=(current bounding box.center)]
\draw[step=1cm, thin] (0,0) grid (5,5);
\drawbigLoutline{0}{1}
\drawbigLoutline{2}{2}
\node at (1.5 ,-0.5) {$t^2x^5$};
\end{tikzpicture}
\quad
\begin{tikzpicture}[scale=0.4,baseline=(current bounding box.center)]
\draw[step=1cm, thin] (0,0) grid (5,5);
\drawbigLoutline{0}{0}
\drawbigLoutline{2}{2}
\node at (1.5 ,-0.5) {$t^2x^5$};
\end{tikzpicture}
\quad
\begin{tikzpicture}[scale=0.4,baseline=(current bounding box.center)]
\draw[step=1cm, thin] (0,0) grid (5,5);
\drawbigLoutline{0}{0}
\drawbigLoutline{1}{1}
\drawbigLoutline{2}{2}
\node at (1.5 ,-0.5) {$t^3x^5$};
\end{tikzpicture}
\\
\begin{tikzpicture}[scale=0.4,baseline=(current bounding box.center)]
\draw[step=1cm, thin] (0,0) grid (6,5);
\drawbigLoutline{0}{0}
\drawbigLoutline{1}{1}
\drawbigLoutline{3}{2}
\node at (1.5 ,-0.5) {$t^3x^6$};
\end{tikzpicture}
\quad
\begin{tikzpicture}[scale=0.4,baseline=(current bounding box.center)]
\draw[step=1cm, thin] (0,0) grid (6,5);
\drawbigLoutline{0}{0}
\drawbigLoutline{2}{1}
\drawbigLoutline{3}{2}
\node at (1.5 ,-0.5) {$t^3x^6$};
\end{tikzpicture}
\\
\begin{tikzpicture}[scale=0.4,baseline=(current bounding box.center)]
\draw[step=1cm, thin] (0,0) grid (7,5);
\drawbigLoutline{0}{0}
\drawbigLoutline{2}{1}
\drawbigLoutline{4}{2}
\node at (1.5 ,-0.5) {$t^3x^7$};
\end{tikzpicture}
\end{center}
Summing all the monomials gives us
\begin{equation*}
\gB_{311,3}(x,t)=(x+3tx^3+3t^2x^4+(3t^2+t^3)x^5+2t^3x^6+t^3x^7).
\end{equation*}
This implies that 
\begin{equation*}
P_n=P_{n-1}+3tP_{n-3}+3t^2P_{n-4}+(3t^2+t^3)P_{n-5}+2t^3P_{n-6}+t^3P_{n-7}.
\end{equation*}
The number of such tilings for $\mu=(3,1,1)$ and $d=3$ gives the sequence: 
\begin{equation}
    1, 1, 1, 4, 10, 20, 41, 90, 199, 431, 928, 2009, 4361,\dotsc.
\end{equation}
This does not appear in the OEIS.
We note that for $n=19$, we have that $P_{19}(t)$ equals 
\begin{align*}
&12 t^{11}+726 t^{10}+15115 t^9+100629 t^8+268452 t^7 \\
&+ 324611 t^6+194706 t^5+61890 t^4+10820 t^3+1038 t^2+51 t+1
\end{align*}
which has a pair of non-real roots $\approx -25.6921 \pm 10.2045 \mathbf{i}$.
\end{example}

\begin{example}[Failure of log-concavity]
Since real-rootedness fails, one can ask if weaker properties hold.
For $\mu=(9,1,1)$, $d=8$ and $n=9$, the polynomial is 
\[
P_{\mu,d,n} = 8 t^7+196 t^6+5208 t^5+9800 t^4+5440 t^3+948 t^2+56 t+1
\]
and this is not even log-concave, as $196^2 < 8\cdot 5208$.
\end{example}

The preceding example is a hook; for the entire hook family, the fault-free
generating function has the following closed form.

\begin{theorem}\label{thm:genFunc}
    Let $\mu = (h,1,1,\dotsc,1)$ be a hook of width $\ell +1$. Then
    \begin{equation}\label{eq:BmudHook}
        \gB_{\mu,d}(x,t) =
        x+x^{\ell+1}\frac{\bigl(1+t(x+x^2+\dotsb+x^\ell)\bigr)^d-1}{x+x^2+\dotsb+x^\ell}.
    \end{equation}
    Equivalently,
    \begin{equation}\label{eq:BmudHookSum}
        \gB_{\mu,d}(x,t) = x+ \sum _{k=1}^d t^k \cdot \binom{d}{k} \cdot x^{\ell +1} \left(x+x^2+\dotsb + x^\ell \right) ^{k-1}.
    \end{equation}
In particular, for $\ell=1$ (i.e.\ $\mu=(h,1)$) and $d=h$, we have
\begin{equation}
\gF_{\mu,d}(x,t) = \frac{1}{1-x(1+xt)^d}.
\end{equation}
\end{theorem}
\begin{proof}
    The number of big tiles, $k$, in a fault-free tiling is between $1$ and $d$.
    We first choose which $k$ rows are occupied by an anchor square.
    The leftmost big tile touches the left-hand side. For each of the remaining $k-1$ tiles we need to decide the horizontal offset with respect to the previous tile. Each such offset has $\ell$ options and increases the width accordingly.
    This gives the sum formula \eqref{eq:BmudHookSum}. Writing
    \[
    S=x+x^2+\dotsb+x^\ell,
    \]
    we obtain
    \[
    \sum_{k=1}^d \binom{d}{k} t^k S^{k-1}=\frac{(1+tS)^d-1}{S},
    \]
    which yields \eqref{eq:BmudHook}. The special case $\ell=1$ gives
    \[
    \gB_{\mu,d}(x,t)=x(1+xt)^d,
    \]
    and the generating function follows from Theorem~\ref{thm:genfuncformula}.
\end{proof}
The following result from \cite[Thm.~2.3]{Liu2007} provides a useful criterion
for establishing real-rootedness via linear recurrences.

\begin{theorem}[Liu--Wang]\label{thm:liuwang}
Let $F$, $f$, $g_1,\dotsc,g_k$ be real polynomials satisfying:
\begin{enumerate}
\item[(a)] $F(t) = a(t)f(t) + b_1(t)g_1(t) + \cdots + b_k(t)g_k(t)$,
where $a(t), b_1(t),\dotsc,b_k(t)$ are real polynomials
and $\deg F \in \{\deg f,\, \deg f + 1\}$.
\item[(b)] $f$ and each $g_j$ are real-rooted, and $g_j \interl f$ for each~$j$.
\item[(c)] $F$ and $g_1,\dotsc,g_k$ have leading coefficients of the same sign.
\end{enumerate}
If $b_j(r) \leq 0$ for each~$j$ and each zero~$r$ of~$f$,
then $F$ is real-rooted and $f \interl F$.
\end{theorem}

\begin{remark}
For two-column hooks $\mu=(h,1)$, Theorem~\ref{thm:interlacing}
proves real-rootedness of $P_{\mu,d,n}(t)$ for every $n$.
With one anchor row, a hook of any width satisfies
\[
\gF_{\mu,1}(x,t)=\frac{1}{1-x-tx^{\ell+1}},
\]
so Theorem~\ref{thm:brandenleite2} gives real-rootedness and consecutive
interlacing in this case as well.
For wider hooks, two anchor rows do not suffice for real-rootedness:
when $\mu=(2,1,1,1,1)$, $d=2$ and $n=18$, we obtain
\[
P_{\mu,2,18}(t)=1+28t+226t^2+560t^3+360t^4+36t^5+t^6,
\]
which is not real-rooted, as verified by an exact root check.

For a hook of width $\ell+1$, Theorem~\ref{thm:genFunc} also gives
$\gB_{\mu,2}(1,1)=\ell+3$.
\end{remark}

\begin{remark}
For general $\mu$, the limiting behavior of the zeros of $P_{\mu,d,n}(t)$
as $n\to\infty$ could perhaps be studied using
the Beraha--Kahane--Weiss theorem~\cite{Beraha1975} on limits
of roots of linear recurrences.
\end{remark}

\begin{example}
Consider $\mu=(3,2,1)$ and $d=3$.
Summing all the monomials gives us:
\begin{equation}
\gB_{321,3}(x,t)=(x+3tx^3+t^2x^4+3t^2x^5+t^3x^7).
\end{equation}
We note that for $t=1$ this is not a unimodal polynomial.
\end{example}

However, for hooks we can prove
the following unimodality result.
\begin{theorem}\label{thm:hookUnimodal}
Suppose $\mu = (h,1,1,\dotsc,1)$ of width $\ell+1$.
Then $\gB_{\mu,d}(x,1)-x$ is unimodal with no internal zeros.
\end{theorem}
\begin{proof}
By Theorem~\ref{thm:genFunc}, with $S = x + x^2 + \cdots + x^\ell$, we have
\[
\gB_{\mu,d}(x,1) - x = x^{\ell+1} \frac{(1+S)^d-1}{S}
= x^{\ell+1} \sum_{j=0}^{d-1} \binom{d}{j+1} S^j.
\]
Let $T = 1 + S = 1 + x + x^2 + \cdots + x^\ell$. Then
\[
\gB_{\mu,d}(x,1) - x = x^{\ell+1} \cdot \frac{T^d - 1}{T - 1}
= x^{\ell+1} \left( 1 + T + T^2 + \cdots + T^{d-1} \right).
\]
It now suffices to show that $1 + T + T^2 + \cdots + T^{d-1}$ is unimodal.
For $\ell\geq2$, this follows from \cite[Thm.~1]{AhmiaBelbachir2018},
applied to $f(t)=1+t+\dotsb+t^{d-1}$, whose coefficients are positive
and nondecreasing. For $\ell=1$, the coefficients of
$((1+x)^d-1)/x$ are $\binom d1,\dotsc,\binom dd$, which are unimodal.
Moreover, $1 + T + T^2 + \cdots + T^{d-1}$ has no internal zeros, and multiplication by $x^{\ell+1}$ preserves both unimodality and the absence of internal zeros.
\end{proof}

We next consider tiles for which the horizontal offsets in fault-free words
are forced.
We say that a Ferrers tile $\mu=(\mu_0,\mu_1,\dotsc,\mu_\ell)$ with $\ell\geq 1$
is \defin{rigid} (with respect to~$d$) if $\mu_\ell=1$ and $\mu_{\ell-1}\geq d$.
In a rigid tile, the only feasible horizontal offset between
consecutive anchor squares in a fault-free word is exactly~$\ell$.

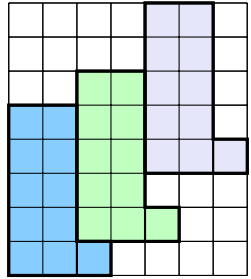
\begin{figure}[ht]
\centering
\begin{tikzpicture}[scale=0.45,baseline=(current bounding box.center)]
\draw[step=1cm, thin] (0,0) grid (7,8);
\draw[thick,fill=skyblue,line width=1pt]
  (0,0) -- ++(3,0) -- ++(0,1) -- ++(-1,0) -- ++(0,4) -- ++(-2,0) -- cycle;
\draw[thick,fill=softgreen,line width=1pt]
  (2,1) -- ++(3,0) -- ++(0,1) -- ++(-1,0) -- ++(0,4) -- ++(-2,0) -- cycle;
\draw[thick,fill=lavender,line width=1pt]
  (4,3) -- ++(3,0) -- ++(0,1) -- ++(-1,0) -- ++(0,4) -- ++(-2,0) -- cycle;
\draw[step=1cm, thin] (0,0) grid (7,8);
\fill[skyblue] (2,0) rectangle (3,1);
\fill[softgreen] (4,1) rectangle (5,2);
\fill[lavender] (6,3) rectangle (7,4);
\draw[thick] (2,0) rectangle (3,1);
\draw[thick] (4,1) rectangle (5,2);
\draw[thick] (6,3) rectangle (7,4);
\draw[black,line width=1.2pt]
  (0,0) -- ++(3,0) -- ++(0,1) -- ++(-1,0) -- ++(0,4) -- ++(-2,0) -- cycle;
\draw[black,line width=1.2pt]
  (2,1) -- ++(3,0) -- ++(0,1) -- ++(-1,0) -- ++(0,4) -- ++(-2,0) -- cycle;
\draw[black,line width=1.2pt]
  (4,3) -- ++(3,0) -- ++(0,1) -- ++(-1,0) -- ++(0,4) -- ++(-2,0) -- cycle;
\end{tikzpicture}
\caption{A fault-free placement of three rigid tiles of shape $\mu=(5,5,1)$.
Here $\ell=2$, and consecutive anchors are separated by exactly two columns,
illustrating the only admissible horizontal offset used in the proof of
the proposition below.}
\label{fig:rigid-five-five-one}
\end{figure}

The conflict graph of a rigid tile is claw-free.
Indeed, for an anchor $(i,r)$, partition its neighbors into two sets:
those strictly to its left together with the anchors in column $i$ below row $r$,
and those strictly to its right together with the anchors in column $i$ above row $r$.
Each set is a clique. Columns at distance less than $\ell$ always conflict;
a neighbor at column $i-\ell$ has row at least $r$, and a neighbor at
$i+\ell$ has row at most $r$.
Thus three pairwise nonadjacent neighbors cannot occur.
Real-rootedness follows from Chudnovsky--Seymour~\cite{Chudnovsky2007};
we give a generating-function proof below as well.

\begin{proposition}\label{prop:BmudRigid}
Let $\mu$ be a rigid tile.
Then
\begin{equation}\label{eq:BmudRigid}
\gB_{\mu,d}(x,t) = x(1+tx^\ell)^d.
\end{equation}
In particular, $\gF_{\mu,d}(x,t) = \frac{1}{1-x(1+tx^\ell)^d}$.
\end{proposition}
\begin{proof}
As illustrated in Figure~\ref{fig:rigid-five-five-one}, a rigid tile can only
be followed by placing the next anchor $\ell$ columns to the right.
If the next anchor were any closer, the second tile would have to start above
a column of height at least $\mu_{\ell-1}\geq d$, which is outside the
$d$ available anchor rows.
At offset $\ell$, the last column has height $\mu_\ell=1$, so the next anchor
only has to move up by one row.
In a fault-free word with $j$ tiles, the anchors therefore occupy columns $1,\ell{+}1,2\ell{+}1,\dotsc$
with strictly increasing rows, and the word has length $j\ell+1$.
Choosing $j$ rows from $\{1,\dotsc,d\}$ gives $\binom{d}{j}$ options,
so $\gB_{\mu,d}(x,t) = \sum_{j=0}^{d}\binom{d}{j}t^j x^{j\ell+1} = x(1+tx^\ell)^d$.
\end{proof}

\begin{theorem}\label{thm:rigidRealRooted}
For any rigid tile $\mu,d$ we have that 
\begin{equation}\label{eq:rigidCoeff}
P_{\mu,d,n}(t) = \sum_{j=0}^{\lfloor dn/(d\ell+1)\rfloor} \binom{d(n-j\ell)}{j}\, t^j,
\end{equation}
and all $P_{\mu,d,n}(t)$ have only real negative zeros.
\end{theorem}
\begin{proof}
The formula~\eqref{eq:rigidCoeff} follows from Proposition~\ref{prop:BmudRigid} by expanding
\[
\frac{1}{1-x(1+tx^\ell)^d} = \sum_{m\geq 0} x^m(1+tx^\ell)^{dm}
\]
and extraction of coefficients.

For real-rootedness, put $r=d\ell+1$ and define
\[
\frac{1}{1-z-tz^r}=\sum_{m\geq0}C_m(t)z^m.
\]
The same coefficient extraction gives
\[
C_m(t)=\sum_{j=0}^{\lfloor m/r\rfloor}
\binom{m-(r-1)j}{j}t^j.
\]
Consequently, $P_{\mu,d,n}=C_{dn}$.
Theorem~\ref{thm:brandenleite2}, with $Q(z)=1-z$, proves the assertion.
\end{proof}
We next prove an interlacing statement that also permits weights.
The fixed-delay case is closely related to Fisk's companion-matrix
argument~\cite[Example 16.10]{Fisk2006x}.
We use the following standard interlacing rules~\cite[Section 3]{Wagner1992}:
if $f\interl g$ and $f\interl h$, then
$f\interl ag+bh$ for $a,b\geq0$ whenever the sum is nonzero;
and, for polynomials with positive leading coefficients and nonpositive zeros,
$f\interl g$ if and only if $g\interl tf$.

\begin{lemma}\label{lem:variableDelay}
Let $j_1,j_2,\dotsc$ be nondecreasing integers with $0\leq j_n<n$,
and let $w_1,w_2,\dotsc$ be nonnegative real numbers.
Define
\[
Q_0(t)=1,\qquad Q_n(t)=Q_{n-1}(t)+w_ntQ_{j_n}(t)\quad(n\geq1).
\]
Then every $Q_n$ has only real negative zeros, and
$Q_i\interl Q_n$ whenever $j_n\leq i\leq n$.
\end{lemma}
\begin{proof}
Every $Q_n$ has nonnegative coefficients and constant term~$1$.
We induct on $n$, the case $n=1$ following from $Q_1=1+w_1t$.
Fix $j_n\leq i<n$. Since $j_{n-1}\leq j_n\leq i$,
induction gives $Q_i\interl Q_{n-1}$.
Also $Q_{j_n}\interl Q_i$: this is reflexivity if $i=j_n$,
and otherwise follows from $j_i\leq j_n<i$ and induction.
The second interlacing rule gives $Q_i\interl tQ_{j_n}$,
and the first gives $Q_i\interl Q_n$.
In particular, $Q_n$ is real-rooted, and its coefficient signs and
constant term exclude nonnegative zeros. The case $i=n$ is reflexivity.
\end{proof}

\begin{theorem}\label{thm:weightedRenewal}
Let $\ell\geq1$ be an integer, and let $R(z)$ be a polynomial of degree $d\geq1$
with nonnegative coefficients, $R(0)=1$, and only real zeros. Define $P_n$ by
\[
\frac{1}{1-xR(tx^\ell)}=\sum_{n\geq0}P_n(t)x^n.
\]
Then $P_n\interl P_{n+s}$ for all integers $n,s\geq0$ such that
$ds\leq d\ell+1$. In particular, every $P_n$ has only real negative zeros.
\end{theorem}
\begin{proof}
We express $P_n$ as a weighted enumerator of subsets with a minimum gap.
Factor $R(z)=\prod_{r=1}^d(1+a_rz)$, where $a_r>0$,
and put $L=d\ell+1$ and $w_i=a_{1+((i-1)\bmod d)}$.
Let $Q_m(t)$ enumerate subsets of $\{1,\dotsc,m\}$ whose consecutive
entries differ by at least $L$, giving entry $i$ weight $w_i t$.
Deleting or retaining the last entry gives
\[
Q_0=1,\qquad
Q_m=Q_{m-1}+w_mtQ_{\max(0,m-L)}\quad(m\geq1).
\]
Lemma~\ref{lem:variableDelay} therefore gives $Q_{i_1}\interl Q_{i_2}$
whenever $0\leq i_1\leq i_2$ and $i_2-i_1\leq L$.

We claim that $P_n=Q_{d(n-\ell)}$ for $n\geq\ell$.
Expanding the generating function gives, for $n-\ell j\geq0$,
\[
[t^j]P_n(t)=[z^j]R(z)^{n-\ell j};
\]
the coefficient is zero when $n-\ell j<0$.
For $j\geq1$, the map
\[
\{y_1<\dotsb<y_j\}\longmapsto
\{y_1,y_2-(L-1),\dotsc,y_j-(j-1)(L-1)\}
\]
is a bijection from the gap-restricted $j$-subsets of
$\{1,\dotsc,d(n-\ell)\}$ to the $j$-subsets of
$\{1,\dotsc,d(n-\ell j)\}$, with both sets empty when the latter
upper bound is less than $j$.
Since $L-1=d\ell$, this bijection preserves every residue modulo $d$
and hence every weight. The weighted enumerator of the latter subsets is
$[z^j]R(z)^{n-\ell j}$ whenever $n-\ell j\geq0$.
The constant terms are both~$1$, proving the claim.

If $n\geq\ell$, the desired comparison follows by taking
$i_1=d(n-\ell)$ and $i_2=i_1+ds$.
If $n<\ell$, then $P_n=1$. When also $n+s<\ell$ both polynomials
are~$1$; otherwise $P_{n+s}=Q_{i_2}$, where
$0\leq i_2=d(n+s-\ell)\leq ds\leq L$, and $Q_0\interl Q_{i_2}$.
\end{proof}

\begin{corollary}\label{cor:rigidInterlacing}
For a rigid tile $\mu$ and integers $n,s\geq0$ with $ds\leq d\ell+1$,
we have $P_{\mu,d,n}(t)\interl P_{\mu,d,n+s}(t)$.
In particular, consecutive tiling polynomials interlace.
\end{corollary}
\begin{proof}
If $d=0$, all the polynomials are~$1$. Otherwise, apply
Theorem~\ref{thm:weightedRenewal} with $R(z)=(1+z)^d$
and use Proposition~\ref{prop:BmudRigid}.
\end{proof}

The gap-restricted subsets in this proof are independent sets in a power
of a path, which is claw-free. Thus their weighted real-rootedness also
follows from Engstr\"om's theorem~\cite[Theorem 2.5]{Engstrom2007}.
Our recurrence proof additionally yields the interlacing relations in
Corollary~\ref{cor:rigidInterlacing}.

\subsection{Anchor words beyond claw-freeness}\label{sec:rectangleBeyondCF}

Lemma~\ref{lem:clawFree} shows that the conflict graph
for $\mu=(k,k)$ with $d\leq k$ is claw-free.
However, when $d$ exceeds this threshold,
real-rootedness can persist far beyond the point
where interlacing fails.

\begin{example}[Delayed failure for $\mu=(2,2,2)$, $d=3$]\label{ex:222delayed}
Consider the anchor words for $\mu=(2,2,2)$ and $d=3$.
These are also the words for $\widehat\mu=(3,2,2)$, and hence count
tilings of a height-5 board by that tile.
For this representative, Lemma~\ref{lem:clawBound} supplies a claw when
$n\geq7$, since $d=3>\mu_{\ell-1}=\mu_1=2$.
They do not count all tilings by $2\times3$ rectangles on a height-4 board.
The fault-free words are
$(\infty)$,
three single tiles (one per row),
$(1,3,\infty,\infty)$ (row~1 chains to row~3 at distance~1),
and $(1,\infty,3,\infty,\infty)$ (row~1 chains to row~3 at distance~2).
This gives the fault-free generating function
\[
\gB_{(2,2,2),3}(x,t) = x + 3tx^3 + t^2x^4 + t^2x^5,
\]
and the recursion
\begin{equation}\label{eq:222recursion}
P_n(t) = P_{n-1}(t) + 3t\,P_{n-3}(t) + t^2\,P_{n-4}(t) + t^2\,P_{n-5}(t).
\end{equation}
The sequence of tiling polynomials $P_0, P_1, P_2, \dotsc$
exhibits the following behavior:
\begin{itemize}
\item For $n\leq 31$:
the polynomial $P_{n-1}$ interlaces $P_n$
whenever $\deg P_n = \deg P_{n-1}+1$.
\item At $n=32$ ($\deg P_{32}=16$):
$P_{31}$ no longer interlaces $P_{32}$,
but both polynomials are still real-rooted.
\item For $32 \leq n \leq 55$: all $P_n$ remain real-rooted
despite the failure of consecutive interlacing.
\item At $n = 56$ ($\deg P_{56}=28$): $P_{56}(t)$ has a pair
of non-real roots, and real-rootedness fails.
\end{itemize}
In particular, real-rootedness persists
for $24$ additional steps after consecutive interlacing breaks down.
\end{example}

Computations for broad families of Ferrers shapes
raise the following problem.

\begin{problem}\label{prob:mu1geqd}
Determine which Ferrers tiles $\mu=(\mu_0,\mu_1,\dotsc,\mu_\ell)$
and strip heights~$d$ have the property that
the tiling polynomials $P_{\mu,d,n}(t)$ are real-rooted for all~$n$.
Which of these families also form a consecutive interlacing sequence,
\[
P_{\mu,d,n-1}(t) \interl P_{\mu,d,n}(t)
\quad\text{for all } n\geq 1?
\]
\end{problem}
The condition $\mu_1\geq d$ is not sufficient: the example
$\mu=(5,5,1,1,1)$, $d=5$, $n=37$ below fails to be real-rooted.
This condition also does not imply claw-freeness. For example,
$\mu=(5,4,3,1)$ with $d=4$ satisfies it, but
Lemma~\ref{lem:clawBound} gives a claw for $n\geq10$.

The next result may be viewed as a fat-L analogue of Theorem~\ref{thm:genFunc}.
\begin{theorem}\label{thm:fatL}
Let $d\geq 1$ and set
\[
\mu=(\underbrace{d,\dotsc,d}_{r},\underbrace{1,\dotsc,1}_{s})
\]
for some $r,s\geq 1$.
Then
\begin{equation}\label{eq:BfatL}
\gB_{\mu,d}(x,t)=x+x^{s}\frac{(1+tx^r(1+x+\dotsb+x^{s-1}))^d-1}{(1+x+\dotsb+x^{s-1})}.
\end{equation}
\end{theorem}
\begin{proof}
Write
\[
S_{r,s}=x^r+x^{r+1}+\dotsb+x^{r+s-1}.
\]
The first part follows the same pattern as Theorem~\ref{thm:genFunc}.
The tile has a block of $r$ columns of height~$d$, followed by a tail of
$s$ columns of height~$1$.
Consider two consecutive anchors in a fault-free word.
Their horizontal offset is at most $r+s-1$, since otherwise there is a fault
line. An offset smaller than $r$ would require a rise of at least $d$ rows,
which is impossible. At every offset $r,\dotsc,r+s-1$, the relevant column
has height~1, so the anchor row must increase strictly.
These conditions are sufficient: for any two nonconsecutive anchors whose
tiles overlap horizontally, the total offset is at least $r$ and their rows
are strictly increasing. Thus each gap contributes
$S_{r,s}=x^r+\dotsb+x^{r+s-1}$.
Thus, if a fault-free word contains $k\geq 1$ big tiles, its anchor rows are
specified by choosing a $k$-element subset of the $d$ possible anchor rows,
which gives $\binom{d}{k}$ choices.
The first big tile contributes width $r+s$, and the remaining $k-1$ tiles
contribute the offsets above. Thus
\[
\gB_{\mu,d}(x,t)
= x+\sum_{k=1}^{d}\binom{d}{k}t^k x^{r+s}S_{r,s}^{k-1}
=x+x^{r+s}\frac{(1+tS_{r,s})^d-1}{S_{r,s}}.
\]
\end{proof}

This family also shows that the condition $\mu_1\geq d$ is not
sufficient on its own. For the fat-L tile with $d=5$, $r=2$, and $s=3$,
that is $\mu=(5,5,1,1,1)$, the polynomial $P_{\mu,d,37}(t)$
is not real-rooted. Similarly, for $d=2$, $r=2$, and $s=4$,
that is $\mu=(2,2,1,1,1,1)$, real-rootedness fails at $n=98$.

\subsection{Real-rootedness failures in the fat-L family}
\label{sec:fatLFailureSearch}

For the fat-L family of Theorem~\ref{thm:fatL}, write
\[
F_{d,r,s}(x,t)
=\frac{1+x+\dotsb+x^{s-1}}
{1-x^s(1+t(x^r+x^{r+1}+\dotsb+x^{r+s-1}))^d}
\]
and
\[
P_{d,r,s,n}(t)=[x^n]F_{d,r,s}(x,t).
\]
Set $A_s(x)=1+x+\dotsb+x^{s-1}$. Expanding the denominator as a
geometric series gives
\[
F_{d,r,s}(x,t)
=A_s(x)\sum_{m\geq0}x^{sm}\bigl(1+tx^rA_s(x)\bigr)^{dm}.
\]
Consequently, we have the closed coefficient formula
\begin{equation}\label{eq:fatLCoefficients}
[t^k]P_{d,r,s,n}(t)
=\sum_{m\geq0}\binom{dm}{k}
[x^{n-sm-rk}]A_s(x)^{k+1}.
\end{equation}
The inner coefficient can also be written explicitly: for $q\geq0$,
\[
[x^q]A_s(x)^{k+1}
=\sum_{j=0}^{\lfloor q/s\rfloor}
(-1)^j\binom{k+1}{j}\binom{q-sj+k}{k}.
\]
Together, these formulas give a closed expression for every coefficient and
make it inexpensive to generate the polynomials directly.

They also determine the exact degree at each width. Indeed, the coefficient
of $t^k$ is positive precisely when
\[
n\geq rk+s\left\lceil\frac{k}{d}\right\rceil.
\]
For the forward implication, use $m\geq\lceil k/d\rceil$ in
\eqref{eq:fatLCoefficients}. Conversely, if the displayed inequality holds,
take $m=\lfloor(n-rk)/s\rfloor$; the remaining exponent lies between $0$
and $s-1$. Hence
\begin{equation}\label{eq:fatLDegree}
\deg_t P_{d,r,s,n}
=\max\left\{k:rk+s\left\lceil\frac{k}{d}\right\rceil\leq n\right\}.
\end{equation}
Equivalently, the last width at which the degree is at most $K$ is
\[
n\leq r(K+1)+s\left\lceil\frac{K+1}{d}\right\rceil-1.
\]

For Problem~\ref{prob:mu1geqd}, we used an exact real-rootedness test for
$2\leq d,r\leq10$, searching $s=2,3,\dotsc$ in order and checking every
polynomial of degree at most~$40$. We omit $d=1$, since the corresponding
conflict graphs are claw-free, so real-rootedness follows from
Chudnovsky--Seymour~\cite{Chudnovsky2007}. We also omit $s=1$, which is
covered by Theorem~\ref{thm:rigidRealRooted}.

Table~\ref{tab:fatLFailureSearch} records the first failures. For each pair
$(d,r)$, we let $s_{\min}$ be the smallest value of $s$ for which a polynomial
of degree at most~$40$ fails to be real-rooted, and let $n_{\min}$ be the
first failing width for this value of~$s$. Thus every smaller $s'\geq2$ was
checked through the width prescribed by \eqref{eq:fatLDegree}.

\begin{table}[!ht]
\centering
\caption{First non-real-rooted fat-L tiling polynomials of degree at
most~$40$, for $2\leq d,r\leq10$. For each $(d,r)$, the two rows give the
smallest failing value $s_{\min}$ and its first failing width $n_{\min}$.}
\label{tab:fatLFailureSearch}
{\small
\setlength{\tabcolsep}{5.5pt}
\renewcommand{\arraystretch}{1.02}
\begin{tabular}{@{}cc*{9}{c}@{}}
\toprule
\multicolumn{2}{c}{} & \multicolumn{9}{c}{$r$} \\
\cmidrule(lr){3-11}
$d$ & & $2$ & $3$ & $4$ & $5$ & $6$ & $7$ & $8$ & $9$ & $10$ \\
\midrule
$2$ & $s_{\min}$ & $4$ & $5$ & $6$ & $8$ & $9$ & $10$ & $11$ & $13$ & $14$ \\
    & $n_{\min}$ & $98$ & $190$ & $284$ & $293$ & $383$ & $464$ & $549$ & $568$ & $658$ \\
\addlinespace[2pt]
$3$ & $s_{\min}$ & $5$ & $7$ & $9$ & $10$ & $12$ & $14$ & $15$ & $17$ & $18$ \\
    & $n_{\min}$ & $111$ & $161$ & $210$ & $323$ & $358$ & $417$ & $504$ & $530$ & $625$ \\
\addlinespace[2pt]
$4$ & $s_{\min}$ & $3$ & $9$ & $11$ & $13$ & $15$ & $17$ & $19$ & $21$ & $23$ \\
    & $n_{\min}$ & $88$ & $164$ & $237$ & $290$ & $343$ & $438$ & $497$ & $556$ & $615$ \\
\addlinespace[2pt]
$5$ & $s_{\min}$ & $3$ & $11$ & $13$ & $15$ & $18$ & $20$ & $22$ & $25$ & $27$ \\
    & $n_{\min}$ & $37$ & $174$ & $253$ & $307$ & $367$ & $421$ & $476$ & $537$ & $591$ \\
\addlinespace[2pt]
$6$ & $s_{\min}$ & $3$ & $4$ & $7$ & $10$ & $14$ & $17$ & $20$ & $23$ & $26$ \\
    & $n_{\min}$ & $28$ & $66$ & $147$ & $233$ & $242$ & $290$ & $337$ & $385$ & $432$ \\
\addlinespace[2pt]
$7$ & $s_{\min}$ & $3$ & $4$ & $5$ & $7$ & $10$ & $12$ & $14$ & $16$ & $18$ \\
    & $n_{\min}$ & $17$ & $47$ & $132$ & $242$ & $196$ & $233$ & $276$ & $372$ & $417$ \\
\addlinespace[2pt]
$8$ & $s_{\min}$ & $3$ & $4$ & $5$ & $6$ & $9$ & $10$ & $12$ & $14$ & $16$ \\
    & $n_{\min}$ & $17$ & $28$ & $74$ & $184$ & $160$ & $258$ & $286$ & $319$ & $277$ \\
\addlinespace[2pt]
$9$ & $s_{\min}$ & $3$ & $4$ & $5$ & $6$ & $8$ & $10$ & $11$ & $13$ & $14$ \\
    & $n_{\min}$ & $17$ & $28$ & $78$ & $102$ & $172$ & $144$ & $303$ & $259$ & $392$ \\
\addlinespace[2pt]
$10$ & $s_{\min}$ & $3$ & $4$ & $5$ & $6$ & $7$ & $9$ & $10$ & $12$ & $13$ \\
     & $n_{\min}$ & $17$ & $28$ & $45$ & $107$ & $201$ & $160$ & $349$ & $280$ & $424$ \\
\bottomrule
\end{tabular}}
\end{table}

For fixed $d$, the smallest observed $s$ generally increases with $r$;
for fixed $r$, it tends to decrease as $d$ grows. The irregular entries show
that this is not an exact law. Moreover, the displayed
values remain cutoff-dependent: increasing the degree bound can only decrease
the smallest recorded~$s$. A numerical Beraha--Kahane--Weiss analysis of the
width recurrences gives the same qualitative picture, with the observed
nonreal zeros lying close to dominant equal-modulus loci. We therefore regard
the table as finite computational evidence rather than a conjectural formula
for the first failure.

\subsection{Counting fault lines}\label{sec:faultLines}

\begin{definition}
It is natural to also count the number of fault lines rather
than the number of big tiles. To that end we introduce the generating function:
\begin{equation} \label{genfuncfaultlines}
\defin{\gG_{\mu,d}(x,t)}\coloneqq \sum_{n\geq0}\sum_{w\in \anchor_{\mu}(d,n)} t^{\fl(w)}x^n
= \sum_{n\geq0} \defin{R_{\mu,d,n}(t)}x^n.
\end{equation}
\end{definition}
Using the same argument as above, we have
\begin{equation}
\gG_{\mu,d} (x,t) = 1+\frac{\gB_{\mu,d}(x,1)}{1-t\cdot \gB_{\mu,d}(x,1)}
\end{equation}
where $\gB_{\mu,d}(x,t)$ is the same as earlier. 
This follows from the fact that a concatenation of $m\geq 1$
fault-free sub-words has exactly $m-1$ fault lines, while the empty word
contributes the initial~$1$.

This formula gives a useful application of Theorem~\ref{thm:brandenleite}.
If $\gB_{\mu,d}(x,1)$ has only real zeros, then
\[
\frac{1}{1-t\cdot\gB_{\mu,d}(x,1)}
\]
has real-rooted coefficient polynomials in~$x$.
For $n\geq 1$, the coefficient of $x^n$ in this series is
$tR_{\mu,d,n}(t)$, so each $R_{\mu,d,n}(t)$ is real-rooted.

\begin{example}
For $\mu=(3,1,1)$, $d=2$ and $n=3$,
the polynomial $R_{\mu,d,n}(t)$ is $2+t^2$.
\end{example}

Similar to Theorem~\ref{thm:twoColumnRealRooted},
the corresponding fault-line generating polynomials are also real-rooted:
\begin{theorem}
For $1\leq d\leq k$ set $R_n(t) \coloneqq R_{(k,k),d,n} (t)$.
The initial values are $R_0=1$, $R_1=1$ and $R_2=t+d$. For $n\geq3$, we have
\begin{equation} 
R_n (t)=t \cdot R_{n-1} (t) + d\cdot t \cdot R_{n-2} (t)
\iff 
\begin{bmatrix}
R_{n-1} \\
R_{n}
\end{bmatrix}
=
\begin{bmatrix}
0 & 1 \\
dt & t
\end{bmatrix}
\begin{bmatrix}
R_{n-2} \\
R_{n-1}
\end{bmatrix},
\end{equation}
and all the $R_n$ have only real zeros.
\end{theorem}

\begin{proof}
A nonempty tiling with $j$ big tiles consists of $n-j$ fault-free blocks,
and hence has $n-j-1$ fault lines. Therefore
\[
R_{(k,k),d,n}(t)=t^{n-1}P_{(k,k),d,n}(t^{-1})\qquad(n\geq1).
\]
The recurrence follows either by appending the last block or from the
recurrence for $P_n$, for $n\geq3$.
Theorem~\ref{thm:twoColumnRealRooted} and the displayed reciprocity
prove real-rootedness: the nonzero roots are reciprocals of roots of $P_n$,
and any additional roots are zero. The empty board contributes $R_0=1$.
\end{proof}

\subsection{Two-column fault-free words}

\begin{example}
For $\mu=(2,2)$ and $d=3$ there are five fault-free anchor words.
The pictures below depict their placements using squares; they do not
exhaust ordinary fault-free tilings of the height-4 board.
Equivalently, the words count all fault-free tilings by $(3,2)$
on a height-5 board.
\begin{center}
\begin{tikzpicture}[scale=0.4,baseline=(current bounding box.center)]
\draw[fill=softgreen,line width= 1pt] (0,0) rectangle (1,1);
\draw[fill=softgreen,line width= 1pt] (0,0) rectangle (1,2);
\draw[fill=softgreen,line width= 1pt] (0,0) rectangle (1,3);
\draw[fill=softgreen,line width= 1pt] (0,0) rectangle (1,4);
\draw[step=1cm, thin] (0,0) grid (1,4);
\end{tikzpicture}
\begin{tikzpicture}[scale=0.4,baseline=(current bounding box.center)]
\draw[fill=peachpink,line width= 1pt] (0,0) rectangle (2,2);
\draw[step=1cm, thin] (0,0) grid (2,4);
\end{tikzpicture}
\begin{tikzpicture}[scale=0.4,baseline=(current bounding box.center)]
\draw[fill=peachpink,line width= 1pt] (0,1) rectangle (2,3);
\draw[step=1cm, thin] (0,0) grid (2,4);
\end{tikzpicture}
\begin{tikzpicture}[scale=0.4,baseline=(current bounding box.center)]
\draw[fill=peachpink,line width= 1pt] (0,2) rectangle (2,4);
\draw[step=1cm, thin] (0,0) grid (2,4);
\end{tikzpicture}
\begin{tikzpicture}[scale=0.4,baseline=(current bounding box.center)]
\draw[fill=skyblue,line width= 1pt] (0,0) rectangle (2,2);
\draw[fill=skyblue,line width= 1pt] (1,2) rectangle (3,4);
\draw[step=1cm, thin] (0,0) grid (3,4);
\end{tikzpicture}
\end{center}
The generating function for the number of tilings is (see \oeis{A097076}),
\begin{equation}
 \frac{1}{1-(x+3x^2+x^3)}.
\end{equation}
This sequence is also studied in \cite{Bodeen2014}, where they consider tilings of a strip of triangles. They have the following five sub-tilings corresponding to the fault-free tilings above. 

\begin{center}
\scalebox{0.625}{%
\begin{tikzpicture}[scale=0.65, line join=round, line cap=round]

\pgfmathsetmacro{\s}{sqrt(3)/2}

\begin{scope}[shift={(0.5,0)}]
  \draw[thick] (0,0) -- (1,0) -- (2,{2*\s}) -- (1,{2*\s}) -- cycle;
  \draw[thick] (1,0) -- (0.5,{\s}) -- (1.5,{\s}) -- (1,{2*\s});
\end{scope}

\begin{scope}[shift={(4,0)}]
  \fill[softpink] (0,0) -- (2,0) -- (1,{2*\s}) -- cycle;
  \draw[thick] (0,0) -- (2,0) -- (3,{2*\s}) -- (1,{2*\s}) -- cycle;
  \draw[thick] (2,0) -- (1,{2*\s});
  \draw[dashed,gray] (0.5,{\s}) -- (2.5,{\s});
  \draw[dashed,gray] (1,0) -- (0.5,{\s});
  \draw[dashed,gray] (1,0) -- (1.5,{\s});
  \draw[dashed,gray] (2,{2*\s}) -- (1.5,{\s});
  \draw[dashed,gray] (2,{2*\s}) -- (2.5,{\s});
\end{scope}

\begin{scope}[shift={(8.5,0)}]
  \fill[softblue] (2,0) -- (1,{2*\s}) -- (3,{2*\s}) -- cycle;
  \draw[thick] (0,0) -- (2,0) -- (3,{2*\s}) -- (1,{2*\s}) -- cycle;
  \draw[thick] (2,0) -- (1,{2*\s});
  \draw[dashed,gray] (0.5,{\s}) -- (2.5,{\s});
  \draw[dashed,gray] (1,0) -- (0.5,{\s});
  \draw[dashed,gray] (1,0) -- (1.5,{\s});
  \draw[dashed,gray] (2,{2*\s}) -- (1.5,{\s});
  \draw[dashed,gray] (2,{2*\s}) -- (2.5,{\s});
\end{scope}

\begin{scope}[shift={(13,0)}]
  \fill[softpink] (0,0) -- (2,0) -- (1,{2*\s}) -- cycle;
  \fill[softblue] (2,0) -- (1,{2*\s}) -- (3,{2*\s}) -- cycle;
  \draw[thick] (0,0) -- (2,0) -- (3,{2*\s}) -- (1,{2*\s}) -- cycle;
  \draw[thick] (2,0) -- (1,{2*\s});
  \draw[dashed,gray] (0.5,{\s}) -- (2.5,{\s});
  \draw[dashed,gray] (1,0) -- (0.5,{\s});
  \draw[dashed,gray] (1,0) -- (1.5,{\s});
  \draw[dashed,gray] (2,{2*\s}) -- (1.5,{\s});
  \draw[dashed,gray] (2,{2*\s}) -- (2.5,{\s});
\end{scope}

\begin{scope}[shift={(17.5,0)}]
  \fill[softpink] (0,0) -- (2,0) -- (1,{2*\s}) -- cycle;
  \fill[softblue] (3,0) -- (2,{2*\s}) -- (4,{2*\s}) -- cycle;
  \draw[thick] (0,0) -- (3,0) -- (4,{2*\s}) -- (1,{2*\s}) -- cycle;
  \draw[thick] (2,0) -- (1,{2*\s});
  \draw[thick] (3,0) -- (2,{2*\s});
  \draw[dashed,gray] (0.5,{\s}) -- (1.5,{\s});
  \draw[dashed,gray] (1,0) -- (0.5,{\s});
  \draw[dashed,gray] (1,0) -- (1.5,{\s});
  \draw[dashed,gray] (1.5,{\s}) -- (2.5,{\s});
  \draw[dashed,gray] (2,0) -- (2.5,{\s});
  \draw[dashed,gray] (2,{2*\s}) -- (1.5,{\s});
  \draw[dashed,gray] (2.5,{\s}) -- (3.5,{\s});
  \draw[dashed,gray] (3,{2*\s}) -- (2.5,{\s});
  \draw[dashed,gray] (3,{2*\s}) -- (3.5,{\s});
\end{scope}

\end{tikzpicture}%
}
\end{center} 
\end{example}

\begin{lemma}\label{lem:countingFaultFreeKK}
Let $\mu = (k,k)$, with $d,k\geq1$.
The number of fault-free anchor words in $\faultfree_\mu (d)$, 
containing exactly $m$ big tiles, is equal to the number of 
integer solutions to the inequalities:
\begin{equation}
 x_1+x_2+x_3+ \cdots + x_m \leq d 
 \text{ where } x_1\geq 1 \text{ and } x_2,x_3, \cdots, x_m \geq k.
\end{equation}
This implies that
\begin{equation}
\gB_{\mu,d}(x,t)= x+\sum_{m=1}^{1+\lfloor(d-1)/k\rfloor} \binom{d+k+m-km-1}{m} t^m x^{m+1}.
\end{equation}
\end{lemma}
\begin{proof}
Given a solution to the inequalities above, we place an anchor square in
column~$j$ and row $x_1+x_2+\dotsb+x_j$.
Conversely, a fault-free word with $m$ big tiles has anchors in columns
$1,2,\dotsc,m$ and width $m+1$. If their rows are
$r_1<\dotsb<r_m$, then non-overlap gives $r_j-r_{j-1}\geq k$.
Hence $x_1=r_1$ and $x_j=r_j-r_{j-1}$ for $j\geq2$ satisfy the
displayed inequalities. These constructions are inverse.
Now counting solutions can be done with the standard bars-and-stars technique, see for example \cite[p.25]{Stanley2011}.
Moreover, unimodality of $\gB_{\mu,d}(x,1)-x$ can be 
proved by showing that the sequence of coefficients $\binom{d+k+m-km-1}{m}$
is log-concave in $m$ for fixed $d$ and $k$.
This can be done by writing the binomial coefficients in terms of 
factorials and verifying the log-concavity inequality.
\end{proof}

\begin{lemma}\label{lem:BdRectRec}
  Fix $k\geq1$, let $d\geq0$, and set $\gB_d \coloneqq \gB_{(k,k),d} (1,t)$.
  Then $\gB_d = \gB_{d-1} + t \cdot \gB_{d-k}$ whenever $d\geq k$,
  with initial values
  \begin{equation}
      \gB_d = 1+dt \quad \text{whenever } 0\leq d \leq k.
  \end{equation}
  Moreover, $\gB_{d-1}\interl \gB_d$ for all $d\geq 1$.
  Consequently, all the $\gB_d (t)$ have only real zeros.
\end{lemma}

\begin{proof}
    Consider a fault-free tiling counted by $\gB_d$.
    If row $d$ does not contain an anchor square,
    this is a fault-free tiling with $d-1$ anchor rows, counted by $\gB_{d-1}$.
    If row $d$ does contain an anchor square,
    the remaining tiles must have anchor rows $\leq d-k$
    (since consecutive anchor rows in a $(k,k)$-tiling differ by at least~$k$),
    giving the term $t\cdot\gB_{d-k}$.

    If $k=1$, the recurrence gives $\gB_d=(1+t)^d$, so the assertion follows.
    Suppose $k\geq2$. We order the state by increasing index:
    \[
    \begin{bmatrix}
    \gB_{d-k+1}\\ \vdots\\ \gB_{d-1}\\ \gB_d
    \end{bmatrix}
    =M_k(t)
    \begin{bmatrix}
    \gB_{d-k}\\ \vdots\\ \gB_{d-2}\\ \gB_{d-1}
    \end{bmatrix},
    \]
    \[
    (M_k)_{ij}=
    \begin{cases}
    1,& j=i+1\text{ and }i<k,\text{ or }i=j=k,\\
    t,& i=k,\ j=1,\\
    0,&\text{otherwise}.
    \end{cases}
    \]
    Thus $M_k(f_1,\dotsc,f_k)^T=(f_2,\dotsc,f_k,tf_1+f_k)^T$.
    We verify Theorem~\ref{thm:ftof}. Put $L=\lambda t+\eta$,
    where $\lambda,\eta>0$. Among the ordered $2\times2$ submatrices,
    the ones for which both pencil polynomials are nonzero are
    \[
    \begin{bmatrix}1&0\\0&1\end{bmatrix},\quad
    \begin{bmatrix}0&1\\t&0\end{bmatrix},\quad
    \begin{bmatrix}0&1\\t&1\end{bmatrix},\quad
    \begin{bmatrix}0&0\\t&1\end{bmatrix}.
    \]
    To see that this list is exhaustive, two rows above the last row
    have their unit entries in increasing columns; when the last row
    is selected, its only nonzero entries are $t$ in column~1 and $1$
    in column~$k$.
    The required comparisons are, respectively,
    \[
    1\interl L,\qquad L\interl t,\qquad
    L+1\interl t,\qquad 1\interl t.
    \]
    These hold because $L$ and $L+1$ have negative zeros and $t$ has zero~0.
    In every remaining case one pencil polynomial is zero and the other
    is zero, a positive constant, or a linear polynomial with nonnegative
    coefficients. Hence $M_k$ preserves $\mathscr F_k^+$.
    The initial vector
    $(\gB_0,\dotsc,\gB_{k-1})=(1,1+t,\dotsc,1+(k-1)t)$
    is interlacing. Iteration proves $\gB_{d-1}\interl\gB_d$.
\end{proof}

\begin{remark}
An alternative proof follows from Theorem~\ref{thm:brandenleite2}:
if $C_m(t)=[z^m](1-z-tz^k)^{-1}$, then
\[
C_m(t)=\sum_{j=0}^{\lfloor m/k\rfloor}\binom{m-(k-1)j}{j}t^j,
\qquad \gB_d=C_{d+k-1}.
\]
Thus the same interlacing follows by taking $Q(z)=1-z$ and $r=k$.
\end{remark}

\begin{problem}
For which Ferrers tiles $\mu$ and parameters $d$ are all coefficient
polynomials $[x^n]\gB_{\mu,d}(x,t)$ real-rooted in $t$?
We have established this for two-column shapes $\mu=(k,k)$
(Lemma~\ref{lem:countingFaultFreeKK} and Lemma~\ref{lem:BdRectRec}),
but the question is open for wider tiles.
\end{problem}

We next study properties of $\gB_{\mu,d}(x,t)-x$.
Questions about unimodality and log-concavity are a well-studied topic in
algebraic combinatorics, see the survey by Stanley~\cite{Stanley1989}.

\begin{problem}
For which $\mu$ and $d$ is $\gB_{\mu,d}(x,1)-x$ unimodal?
When are there no internal zeros?

We answered this question for hooks, the case $\mu=(k,k)$
is proved in Lemma~\ref{lem:countingFaultFreeKK}.
\end{problem}

\subsection{Fault lines in dense tilings}\label{sec:dense}
\begin{definition}
    A tiling is called \defin{dense} if every column contains part of a big tile.
    A column with anchor symbol $\infty$ need not be empty. 
    The generating function for dense tilings is then
	    \begin{equation}
	     \defin{\gD_{\mu,d}(x,t)}\coloneqq
         \frac{1}{1+x-\gB_{\mu,d}(x,t)}
	    \end{equation}
	    by disallowing the empty column in \eqref{genfuncformula}. 

	    We will now focus on the closely related fault-line-refined dense
	    generating function $\gH_{\mu,d}(x,s)$ defined by 
	    \begin{equation}
	     \defin{\gH_{\mu,d}(x,s)} \coloneqq \frac{1}{1-s(\gB_{\mu,d}(x,1)-x)} =1+ \sum_{n\geq 1}x^n \sum_{T} s^{\mathrm{faultlines}(T)+1},
	    \end{equation}
    where the second sum ranges over all \emph{dense} tilings of the $(\mu_{0}+d-1) {\times} n$ board. 
\end{definition}

\begin{example}
Let us consider $\mu=(2,1)$ and $d=2$. Then $\gH_{21,2}(x,s)=\frac{1}{1-s(2x^2+x^3)}$. If we let $Q_{n}(s) \coloneqq [x^n]\, \gH_{21,2}(x,s)$, then 
 \begin{equation} \label{eq:denserecursion}
    Q_{n}(s)=2sQ_{n-2}(s)+sQ_{n-3}(s),
\end{equation}
for $n\geq3$, with $Q_0=1$, $Q_1=0$, $Q_2=2s$.
For example, $Q_{6}(s)=8s^3+s^2$ as seen in Figure \ref{fig:densetilings}. 

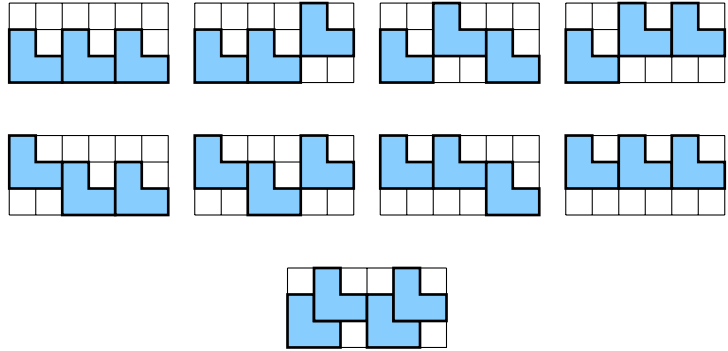
\begin{figure} [!ht]
    \centering
\begin{tikzpicture}[scale=0.35]
\begin{scope}[shift={(0,0)}]
\draw[step=1cm,thin] (0,0) grid (6,3);
\drawLoutline{0}{0}
\drawLoutline{2}{0}
\drawLoutline{4}{0}
\end{scope}

\begin{scope}[shift={(7,0)}]
\draw[step=1cm,thin] (0,0) grid (6,3);
\drawLoutline{0}{0}
\drawLoutline{2}{0}
\drawLoutline{4}{1}
\end{scope}

\begin{scope}[shift={(14,0)}]
\draw[step=1cm,thin] (0,0) grid (6,3);
\drawLoutline{0}{0}
\drawLoutline{2}{1}
\drawLoutline{4}{0}
\end{scope}

\begin{scope}[shift={(21,0)}]
\draw[step=1cm,thin] (0,0) grid (6,3);
\drawLoutline{0}{0}
\drawLoutline{2}{1}
\drawLoutline{4}{1}
\end{scope}

\begin{scope}[shift={(0,-5)}]
\draw[step=1cm,thin] (0,0) grid (6,3);
\drawLoutline{0}{1}
\drawLoutline{2}{0}
\drawLoutline{4}{0}
\end{scope}

\begin{scope}[shift={(7,-5)}]
\draw[step=1cm,thin] (0,0) grid (6,3);
\drawLoutline{0}{1}
\drawLoutline{2}{0}
\drawLoutline{4}{1}
\end{scope}

\begin{scope}[shift={(14,-5)}]
\draw[step=1cm,thin] (0,0) grid (6,3);
\drawLoutline{0}{1}
\drawLoutline{2}{1}
\drawLoutline{4}{0}
\end{scope}

\begin{scope}[shift={(21,-5)}]
\draw[step=1cm,thin] (0,0) grid (6,3);
\drawLoutline{0}{1}
\drawLoutline{2}{1}
\drawLoutline{4}{1}
\end{scope}

\begin{scope} [shift={(10.5,-10)}]
\draw[step=1cm, thin] (0,0) grid (6,3);
\drawLoutline{0}{0}
\drawLoutline{1}{1}
\drawLoutline{3}{0}
\drawLoutline{4}{1}
\end{scope}
\end{tikzpicture}
\caption{Dense tilings.}\label{fig:densetilings}
\end{figure}
We note that the number of dense tilings counted by $Q_{n}(1)$ is equal to the sequence \oeis{A008346}, which is the 
Pell--Padovan's sequence, see \cite{Bilgici2013}.

The recursion in \eqref{eq:denserecursion} can be expressed as 
\begin{equation}
    \begin{bmatrix}
        Q_{n-2}\\
        Q_{n-1}\\
        Q_{n}\\
    \end{bmatrix} =
    \begin{bmatrix}
        0 & 1 & 0\\
        0 & 0 & 1\\
        s & 2s & 0\\
    \end{bmatrix}
    \begin{bmatrix}
        Q_{n-3}\\
        Q_{n-2}\\
        Q_{n-1}\\
    \end{bmatrix}.
\end{equation}
The initial vector $(1,0,2s)^T$ is interlacing, using the convention for zero.
To apply Theorem~\ref{thm:ftof}, let $\lambda,\eta>0$.
Apart from comparisons with zero, its $2\times2$ conditions reduce to
\[
1\interl\lambda s+\eta,\qquad
(\lambda+c)s+\eta\interl as,\qquad 2s\interl s,
\]
where $c\in\{0,2\}$ and $a\in\{1,2\}$.
The linear polynomial on the left in the middle comparison has a negative
zero, whereas $as$ has zero at~0. Thus all conditions hold, and every
$Q_n$ is real-rooted.
\end{example}

\begin{remark}
To mark empty columns by $u$, we mark the singleton fault-free block
$(\infty)$. All other fault-free blocks have no empty columns. Thus
\begin{gather*}
\gB_{\mu,d}(x,t,u)=ux+\gB_{\mu,d}(x,t)-x,
\\
\gF_{\mu,d}(x,t,u)=\frac{1}{1-ux-(\gB_{\mu,d}(x,t)-x)}.
\end{gather*}
Setting $u=0$ gives the dense-tiling generating function.
\end{remark}

\section{A row-refined tiling polynomial} \label{sec:rowRefined}

Another natural variation of the tiling polynomial is obtained by distinguishing tiles
by the row of their anchor square.
For anchor words in $\anchor_\mu(d,n)$, define
\[
  \defin{S_{\mu,d,n}(s_1,\dotsc,s_d)} \coloneqq \sum_{w\in\anchor_\mu(d,n)} \prod_{\substack{1\leq i\leq n \\ a_i\neq\infty}} s_{a_i},
\]
so that $s_r$ marks tiles with anchor in row $r$.
Note that $S_{\mu,d,n}(1,\dotsc,1)=P_{\mu,d,n}(1)$ (the total number of tilings).

\begin{theorem}\label{thm:row1RealRooted}
For $\mu=(k,k)$ and $d\leq k$, the generating function for the row-refined polynomials is
\begin{equation}\label{eq:row1GF}
  \sum_{n\geq 0} S_{(k,k),d,n}(s_1,\dotsc,s_d)\,x^n = \frac{1}{1 - x - (s_1+s_2+\cdots+s_d)\,x^2}.
\end{equation}
Equivalently,
\begin{equation}\label{eq:row1closed}
  S_{(k,k),d,n}(s_1,\dotsc,s_d) = \sum_{j=0}^{\lfloor n/2\rfloor} \binom{n-j}{j}(s_1+s_2+\cdots+s_d)^j.
\end{equation}
In particular, setting $s_1=s$ and the remaining variables to 1, the univariate polynomial
$S_{(k,k),d,n}(s,1,\dotsc,1)= \sum_{j=0}^{\lfloor n/2\rfloor} \binom{n-j}{j}(s+d-1)^j$
is real-rooted for all $n\geq 0$.
\end{theorem}

\begin{proof}
For $d\leq k$, the anchor word condition forces any two consecutive non-$\infty$ entries
to differ by at least $k\geq d$, which is impossible.
Thus every non-$\infty$ entry must be followed immediately by $\infty$:
the big tiles are isolated in the anchor word.

An anchor word of length $n$ ending in $\infty$ with $j$ isolated big tiles is constructed
by choosing $j$ non-adjacent positions in $\{1,\ldots,n-1\}$ for the tiles (the remaining position $n$ is forced to $\infty$);
the number of such choices is $\binom{n-j}{j}$.
Each tile with anchor in row $r$ contributes weight $s_r$.
Since the tiles are independent, each contributes a factor of $(s_1+s_2+\cdots+s_d)$.
Summing over $j$ gives \eqref{eq:row1closed}, from which \eqref{eq:row1GF} follows.

For the univariate specialization $s_1=s$, $s_2=\cdots=s_d=1$,
the polynomial is $M_n(s+d-1)$, where
$M_n(u)=\sum_j\binom{n-j}{j}u^j$ is the matching generating polynomial
of the path on $n$ vertices.
It also equals $P_{(1,1),1,n}(u)$, so Theorem~\ref{thm:twoColumnRealRooted}
proves real-rootedness. Translation of the variable preserves this property.
The signed matching polynomial uses the convention
$\mathfrak m_n(z)=z^n M_n(-z^{-2})$.
This change of variables also relates the assertion to the classical
monomer--dimer theorem of Heilmann and Lieb~\cite{HeilmannLieb1972}.
\end{proof}

\begin{example}
Setting $s=1$ recovers the total tiling count:
\begin{equation}
  S_{(k,k),d,n}(1) = \sum_{j=0}^{\lfloor n/2\rfloor}\binom{n-j}{j}d^j = [x^n]\,\frac{1}{1-x-dx^2},
\end{equation}
giving the sequences \oeis{A001045} ($d=2$, Jacobsthal), \oeis{A006130} ($d=3$),
and \oeis{A006131} ($d=4$).
The coefficient of $s^1$ in $S_{(k,k),d,n}(s)$ (counting tilings with exactly one row-$1$ anchor)
equals $[x^{n-2}]\bigl(1/(1-x-(d-1)x^2)\bigr)^2$, the self-convolution of the $d{-}1$ sequence;
for $d=2$ this is \oeis{A001629} (self-convolution of Fibonacci) and for $d=3$ it is \oeis{A073371}
(self-convolution of the Jacobsthal sequence).
\end{example}

\begin{theorem}\label{thm:row1Hook}
Let $\mu = (h,1,1,\dotsc,1)$ be a hook of width $\ell+1$.
Let $e_k(\mathbf{s}) = e_k(s_1,\dotsc,s_d)$ denote the $k$th elementary symmetric polynomial.
The row-refined fault-free generating function is
\begin{equation}\label{eq:rowHookB}
\gB_{\mu,d}(x,\mathbf{s})
= x + \sum_{k=1}^{d} e_k(\mathbf{s})\cdot x^{\ell+1}(x+x^2+\dotsb+x^\ell)^{k-1},
\end{equation}
and the row-refined tiling generating function is
$\sum_{n\geq 0} S_{\mu,d,n}(\mathbf{s})\,x^n = 1/(1-\gB_{\mu,d}(x,\mathbf{s}))$.
\end{theorem}
\begin{proof}
Since $\mu_1=\dotsb=\mu_\ell=1$, the anchor word constraint at every offset
$k=1,\dotsc,\ell$ is $a_{i+k}\geq a_i+1$.
In a fault-free word with $m$ tiles, the anchor rows therefore form
a strictly increasing sequence $r_1 < r_2 < \cdots < r_m$ in $\{1,\dotsc,d\}$,
contributing the weight $s_{r_1}\cdots s_{r_m} = e_m$ summed over all such choices.
The horizontal offset between consecutive tiles can be any value in $\{1,\dotsc,\ell\}$,
giving $\ell$ choices per gap and $(x+x^2+\dotsb+x^\ell)^{m-1}$ for the $m-1$ gaps.
The first tile occupies $\ell+1$ columns, yielding the formula.
The generating function follows from Theorem~\ref{thm:genfuncformula}.
\end{proof}

When the horizontal offset is forced to be unique, the formula simplifies to a product.

\begin{theorem}\label{thm:row1Lpoly}
Let $\mu = (\mu_0,\mu_1,\dotsc,\mu_\ell)$ with $\ell \geq 1$, $\mu_\ell = 1$, and $\mu_{\ell-1}\geq d$.
The row-refined generating function is
\begin{equation}\label{eq:row1LpolyGF}
  \sum_{n\geq 0} S_{\mu,d,n}(s_1,\dotsc,s_d)\,x^n = \frac{1}{1 - x\prod_{r=1}^{d}(1+s_r x^\ell)}.
\end{equation}
In particular, setting $s_1=s$ and $s_2=\cdots=s_d=1$, we obtain
$\frac{1}{1 - x(1+x^\ell)^{d-1}(1+sx^\ell)}$.
\end{theorem}
\begin{proof}
By Proposition~\ref{prop:BmudRigid}, the condition $\mu_{\ell-1}\geq d$
	forces the only admissible horizontal offset to be exactly~$\ell$.
The row-refined fault-free generating function from Theorem~\ref{thm:row1Hook}
therefore has $(x+\dotsb+x^\ell)^{k-1}$ replaced by $(x^\ell)^{k-1}$, giving
\[
\gB_{\mu,d}(x,\mathbf{s})
= x + \sum_{k=1}^{d} e_k(\mathbf{s})\,x^{k\ell+1}
= x\prod_{r=1}^{d}(1+s_r x^\ell).
\qedhere
\]
\end{proof}

\begin{example}
For $\mu=(h,1)$ (so $\ell=1$) and $d=2$,
the generating function is $1/(1-x(1+x)(1+sx))$,
and $S_n(s) \coloneqq S_{\mu,d,n}(s,1)$ satisfies
$S_n = S_{n-1} + (s+1)S_{n-2} + sS_{n-3}$.
Setting $s=1$ recovers the sequence \oeis{A002478};
the coefficient of $s^1$ gives \oeis{A023610}.
\end{example}

\subsection{Real-rootedness across polynomial variants}\label{sec:variants}

Table~\ref{tab:variants} compares real-rootedness of four polynomial families
associated with the tiling:
the standard tiling polynomial $P_n(t)$;
the row-$1$ specialization $S_n(s)\coloneqq S_{\mu,d,n}(s,1,\dotsc,1)$,
which tracks the number of tiles anchored in row~$1$;
the row-$d$ specialization $T_n(s)\coloneqq S_{\mu,d,n}(1,\dotsc,1,s)$;
and the dense tiling polynomial
$Q_n(s)=[x^n]\,\frac{1}{1-s(\gB_{\mu,d}(x,1)-x)}$.

The two row specializations satisfy $S_n=T_n$ for every $n$.
Indeed, reflecting the allowed anchor columns and replacing each anchor row
$a$ by $d+1-a$ preserves the anchor inequalities and exchanges rows~$1$
and~$d$. We retain both columns in the table to display the two statistics.

\begin{table}[!ht]
\centering
\caption{Real-rootedness of polynomial variants
(verified for $n\leq 500$ by exact computation or the cited theorems).
The entries concern anchor words; the original tile has the stated physical
interpretation only when $d\leq\mu_0$.
A check mark~\checkmark\ means real-rooted for all tested~$n$;
a cross~$\times_k$ means first failure at $n=k$.
The first row holds for every $n$ by Theorems~\ref{thm:twoColumnRealRooted}
and~\ref{thm:row1RealRooted}; its dense polynomials are zero or monomials.
For $(h,1)$, real-rootedness of $P_n$ holds for every $d,n$ by
Theorem~\ref{thm:interlacing}.}
\label{tab:variants}
\renewcommand{\arraystretch}{1.25}
\begin{tabular}{@{\hspace{1em}}l@{\hspace{2em}}c@{\hspace{2em}}c@{\hspace{2em}}c@{\hspace{2em}}c@{\hspace{2em}}c@{\hspace{1em}}}
\toprule
$\mu$ & $d$ & \shortstack{$P_n(t)$\\[-2pt]\scriptsize tiling} & \shortstack{$S_n(s)$\\[-2pt]\scriptsize row-1} & \shortstack{$T_n(s)$\\[-2pt]\scriptsize row-$d$} & \shortstack{$Q_n(s)$\\[-2pt]\scriptsize dense} \\
\midrule
\multicolumn{6}{@{\hspace{1em}}l}{\emph{Two-column tiles}} \\
$(k,k)$ & $\leq k$ &  \checkmark & \checkmark & \checkmark & \checkmark \\
$(2,2)$ & $3$ & \checkmark & \checkmark & \checkmark & \checkmark \\
$(2,2)$ & $4$ & \checkmark & $\times_{22}$ & $\times_{22}$ & \checkmark \\
$(3,3)$ & $4$ & \checkmark & $\times_{30}$ & $\times_{30}$ & \checkmark \\
$(h,1)$ & $d\leq5$ & \checkmark & \checkmark & \checkmark & \checkmark \\
\midrule
\multicolumn{6}{@{\hspace{1em}}l}{\emph{Three-column tiles}} \\
$(3,1,1)$ & $3$ & $\times_{19}$ & \checkmark & \checkmark & \checkmark \\
$(h,1,1)$ & $3$ & $\times_{19}$ & \checkmark & \checkmark & \checkmark \\
$(2,2,2)$ & $3$ & $\times_{56}$ & \checkmark & \checkmark & \checkmark \\
$(h,2,2)$ & $3$ & $\times_{56}$ & \checkmark & \checkmark & \checkmark \\
$(3,2,1)$ & $3$ & \checkmark & \checkmark & \checkmark & $\times_{17}$ \\
$(3,3,2)$ & $3$ & \checkmark & \checkmark & \checkmark & $\times_{15}$ \\
$(3,3,3)$ & $4$ & $\times_{24}$ & $\times_{180}$ & $\times_{180}$ & \checkmark \\
$(4,4,4)$ & $5$ & $\times_{16}$ & $\times_{45}$ & $\times_{45}$ & \checkmark \\
\midrule
\multicolumn{6}{@{\hspace{1em}}l}{\emph{Four-column tiles}} \\
$(3,1,1,1)$ & $3$ & $\times_{18}$ & \checkmark & \checkmark & \checkmark \\
$(2,2,2,2)$ & $3$ & $\times_{36}$ & \checkmark & \checkmark & \checkmark \\
$(h,2,2,1)$ & $4$ & \checkmark & \checkmark & \checkmark & \checkmark \\
$(4,3,2,1)$ & $4$ & $\times_{41}$ & \checkmark & \checkmark & $\times_{84}$ \\
\bottomrule
\end{tabular}
\renewcommand{\arraystretch}{1.0}
\end{table}

Several of the integer sequences $P_n(1)$, $Q_n(1)$, and $S_n(0)$
(tilings avoiding row~$1$, equivalently $P_{\mu,d-1,n}(1)$)
appear in the OEIS or are new;
see Table~\ref{tab:oeis}.

\begin{table}[ht!]
\centering
\caption{OEIS connections for anchor-word counts (ordinary tiling counts
for the original tile when $d\leq\mu_0$).
Sequences marked ``New'' do not appear in the OEIS
at the time of writing.}
\label{tab:oeis}
\renewcommand{\arraystretch}{1.3}
\begin{tabular}{@{\hspace{1em}}l@{\hspace{2em}}l@{\hspace{2em}}c@{\hspace{2em}}l@{\hspace{1em}}}
\toprule
Sequence & $\mu$, $d$ & GF & OEIS \\
\midrule
\multicolumn{4}{@{\hspace{1em}}l}{\emph{Total tilings $P_n(1)$}} \\
$P_n(1)$ & $(h,1),\; d{=}2$ & $\frac{1}{1-x(1+x)^2}$ & \oeis{A002478} \\
$P_n(1)$ & $(h,1),\; d{=}3$ & $\frac{1}{1-x(1+x)^3}$ & \oeis{A099234} \\
$P_n(1)$ & $(2,2),\; d{=}3$ & $\frac{1}{1-x-3x^2-x^3}$ & \oeis{A097076} \\
$P_n(1)$ & $(3,2,1),\; d{=}2$ & $\frac{1}{1-x-2x^3-x^5}$ & \oeis{A193147} \\
$P_n(1)$ & $(3,2,1),\; d{=}3$ & --- & New \\
$P_n(1)$ & $(h,2,2,1),\; d{=}4$ & --- & New \\
$P_n(1)$ & $(3,3,2),\; d{=}3$ & --- & New \\
\midrule
\multicolumn{4}{@{\hspace{1em}}l}{\emph{Dense tilings $Q_n(1)$}} \\
$Q_n(1)$ & $(h,1),\; d{=}2$ & $\frac{1}{1-2x^2-x^3}$ & \oeis{A008346} \\
$Q_n(1)$ & $(3,1,1),\; d{=}3$ & --- & New \\
$Q_n(1)$ & $(h,2,2,1),\; d{=}4$ & --- & New \\
\bottomrule
\end{tabular}
\renewcommand{\arraystretch}{1.0}
\end{table}

\section{Tilings of a cylinder}\label{sec:cylindric}

We identify the left and right boundaries of the board and label the
columns $1,\dotsc,n$. Positions are not identified under rotation;
a tile placement includes its anchor column.
For $n\geq\ell+1$, a \defin{cylindric anchor word} satisfies
\[
a_i\ne\infty\ \Longrightarrow\ a_{i+k}\geq a_i+\mu_k
\qquad(1\leq i\leq n,\ 1\leq k\leq\ell),
\]
where indices are taken modulo $n$. There is no terminal-letter condition.
We require a tile to occupy distinct columns, so for $0<n<\ell+1$
only the all-$\infty$ word is allowed. We exclude self-overlapping tiles.
Let \defin{$\anchorc_\mu(d,n)$} denote this set and let
$\anchorc_\mu(d,0)=\{\varepsilon\}$.
As before, the physical interpretation uses $d\leq\mu_0$, or the
representative $\widehat\mu$.

\begin{lemma}
Every nonempty cylindric anchor word has a fault line.
\end{lemma}
\begin{proof}
The all-$\infty$ word has a fault line at every cut.
Otherwise, list the finite entries in cyclic order. If the gap between two
successive anchors is greater than $\ell$, there is a cut between them
crossed by no tile. If there is no such gap, the cyclic inequalities force
successive anchor rows to increase strictly, including from the last
anchor back to the first. This~is~impossible.
\end{proof}
 
From the previous lemma it follows that every cylindric anchor word is a cyclic shift of 
a regular anchor word. This property allows us to make the following observations.
\bigskip

Let $w \in \anchor_\mu (d,n)$ be an anchor word where $w$ is expressed as the concatenation
\begin{equation}
w=\beta_1,\beta_2, \dots, \beta_m, \text{where each $\beta_i$ is in $\faultfree_\mu(d)$}.
\end{equation}

Let $k$ be the width of $\beta_1$. Cyclical shifting of $w$ by $0,1,\dots, k-1$ steps
to the left, gives $k$ different elements in $\anchorc_\mu (d,n)$.
For a nonempty cylindric word, choose $\beta_1$ to be the unique fault-free
block containing column~1. Its width determines the possible positions of
column~1 within it. The remaining blocks follow in cyclic order. This gives
a unique decomposition, including for periodic words.

The generating polynomial for all the shifted words in $\faultfree_\mu(d)$ is therefore, $x \cdot \frac{\partial}{\partial x} \gB_{\mu,d}(x,t)$.
We can now state the cylindric analog of the generating function in \eqref{genfuncformula}.
\begin{theorem}
Let \defin{$\gC_{\mu,d}(x,t)$} be defined as 
\begin{equation}
\gC_{\mu,d}(x,t)\coloneqq \sum_{n\geq0}\sum_{w\in \anchorc_{\mu}(d,n)} t^{\bt(w)}x^n. 
\end{equation}
Then we have
\begin{equation}
\gC_{\mu,d}(x,t) = 1+ \frac{x \cdot \frac{\partial}{\partial x} \gB_{\mu,d}(x,t)}{1-\gB_{\mu,d}(x,t)}. 
\end{equation}
\end{theorem}

\begin{corollary}
The generating function for $\mu = (h,1)$ and $d=h$ for cylindric words is
\begin{equation}
\gC_{\mu,d}(x,t) = 1+ \frac{x \cdot (dtx(1+tx)^{d-1}+(1+tx)^d)}{1-x(1+xt)^d}.
\end{equation}
\end{corollary}

\begin{example}
Below are the sequences for $|\anchorc_\mu(d,n)|$ with $\mu=(h,1)$ and $d=h=1,\dotsc,8$.

\begin{table}[!ht]
\centering
\caption{Cylindric tiling counts $|\anchorc_{(h,1)}(h,n)|$ for small~$h=d$.}\label{tab:cylindricCounts}
\begin{tabular}{@{}cll@{}}
\toprule
$d$ & \textbf{Sequence} ($n=1,2,3,\dotsc$) & \textbf{OEIS} \\
\midrule
1 & 1, 3, 4, 7, 11, 18, 29, 47, \dots & \oeis{A000032} \\
2 & 1, 5, 10, 21, 46, 98, 211, 453, \dots & \oeis{A286910} \\
3 & 1, 7, 19, 47, 126, 331, 869, \dots & New \\
4 & 1, 9, 31, 89, 276, 855, 2626, \dots & New \\
5 & 1, 11, 46, 151, 526, 1862, 6518, \dots & New \\
6 & 1, 13, 64, 237, 911, 3604, 14113, \dots & New \\
7 & 1, 15, 85, 351, 1471, 6399, 27637, \dots & New \\
8 & 1, 17, 109, 497, 2251, 10637, 50107, \dots & New \\
\bottomrule
\end{tabular}
\end{table}
\end{example}

\begin{theorem}\label{thm:cylindricRealRooted}
The generating polynomial for $\mu = (h,1)$ and $d=h$ for cylindric words of length $n$, that is 
\begin{equation}
    \sum_{w\in \anchorc_{\mu}(d,n)} t^{\bt(w)}
\end{equation}
has only real zeros.
\end{theorem}
\begin{proof}
For $n=0,1$ the polynomial is~1. For $n=2$ the conflict graph is complete:
two anchors in different columns would require both rows to be strictly
larger than the other. Suppose $n\geq3$ and fix an anchor $(i,r)$.
Its neighbors in column $i-1$ have rows at least $r$, and those in column
$i+1$ have rows at most $r$, with columns interpreted cyclically.
The former neighbors together with anchors in column $i$ below $r$ form
a clique. The latter neighbors together with anchors in column $i$ above
$r$ form another clique. These two cliques cover the neighborhood,
so the graph is claw-free. Apply Chudnovsky--Seymour~\cite{Chudnovsky2007}.
\end{proof}

\section*{Acknowledgments}

We thank the anonymous referee for an exceptionally careful reading and an
insightful, constructive report whose suggestions substantially improved the
paper.

OpenAI Codex was used for LaTeX assistance, including generating TikZ figure macros,
table formatting, and checking the manuscript for possible typos and grammatical issues.
During the revision, Codex also suggested the Br\"and\'en--Leite proof
for rigid tiles, the weighted recurrence interlacing argument, and the increasing-order matrix proof for fault-free words,
and assisted with the mathematical corrections and exact checks.

\bibliographystyle{plain}
\bibliography{partial-tilings-bibliography}

\end{document}